\documentclass[11pt]{article}
\usepackage[utf8]{inputenc}
\usepackage{amsmath,amsfonts,amsthm,amssymb,bbold}
\usepackage{bm,bbm}
\usepackage[margin=1.8cm]{geometry}
\usepackage[dvipsnames]{xcolor}
\usepackage{graphicx}
\usepackage{tikz}
\usepackage{verbatim}
\usepackage{authblk}
\usepackage{accents}
\usepackage{placeins}
\usepackage{microtype}
\usepackage[colorlinks=true]{hyperref}
\usepackage[capitalise]{cleveref}
\newcommand{\ppone}{\bm{+1}}
\newcommand{\mmone}{\bm{-1}}
\newcommand{\pmone}{\bm{\pm 1}}
\newcommand{\mpone}{\bm{\mp 1}}
\newcommand{\cX}{\mathcal{X}}
\newcommand{\cQ}{\mathcal{Q}}

\newtheorem{theorem}{Theorem}[section]
\newtheorem{lemma}[theorem]{Lemma}
\newtheorem{proposition}[theorem]{Proposition}
\newtheorem{definition}[theorem]{Definition}
\newtheorem{remark}[theorem]{Remark}

\newtheorem{corollary}[theorem]{Corollary}

\numberwithin{equation}{section}
\numberwithin{theorem}{section}
\definecolor{bleudefrance}{rgb}{0.19, 0.55, 0.91}
\definecolor{islamicgreen}{rgb}{0.0, 0.56, 0.0}

\newcommand{\sib}[1]{{\color{bleudefrance} [Simo: #1]}}

\title{Ising model on clustered networks: A model for opinion dynamics}

\author[1,2]{Simone Baldassarri}
\affil[1]{Universit\`{a} degli Studi di Firenze, Firenze, Italy}
\affil[2]{Aix-Marseille Université, Marseille, France}
\author[3]{Anna Gallo}
\affil[3]{IMT School for Advanced Studies, Lucca, Italy}
\author[4]{Vanessa Jacquier}
\affil[4]{Scuola Normale Superiore, Pisa, Italy}
\author[5]{Alessandro Zocca}
\affil[5]{Vrije Universiteit Amsterdam, Amsterdam, The Netherlands}
\date{\today}
\begin{document}

\maketitle

\begin{abstract}
We study opinion dynamics on networks with a nontrivial community structure, assuming individuals can update their binary opinion as the result of the interactions with an external influence with strength $h\in [0,1]$ and with other individuals in the network. To model such dynamics, we consider the Ising model with an external magnetic field on a family of finite networks with a clustered structure. Assuming a unit strength for the interactions inside each community, we assume that the strength of interaction across different communities is described by a scalar $\epsilon \in [-1,1]$, which allows a weaker but possibly antagonistic effect between communities. We are interested in the stochastic evolution of this system described by a Glauber-type dynamics parameterized by the inverse temperature $\beta$. We focus on the low-temperature regime $\beta\rightarrow\infty$, in which homogeneous opinion patterns prevail and, as such, it takes the network a long time to fully change opinion. We investigate the different metastable and stable states of this opinion dynamics model and how they depend on the values of the parameters $\epsilon$ and $h$. More precisely, using tools from statistical physics, we derive rigorous estimates in probability, expectation, and law for the first hitting time between metastable (or stable) states and (other) stable states, together with tight bounds on the mixing time and spectral gap of the Markov chain describing the network dynamics. Lastly, we provide a full characterization of the critical configurations for the dynamics, i.e., those which are visited with high probability along the transitions of interest.
\end{abstract}

\medskip

		


\section{Introduction}
The Ising model was originally introduced to study ferromagnetism \cite{Ising1925} and is probably one of the most studied models in statistical physics. The spins are arranged in a given graph structure and each of them can be in one of two states $+1$ (``upwards'') or $-1$ (``downwards''). These spins interact with each other in a stochastic fashion but each spin has the tendency to align with its neighbors as this result in low-energy configurations for the system. In the statistical physics literature, researchers have primarily considered the Ising model on lattice structures or complete graphs (in which case it is also known as Curie-Weiss model), for a comprehensive historical perspective of this model we refer to \cite{DuminilCopin2022}.

Later, the Ising model has also been used to study a wide range of physical and nonphysical phenomena, and, in particular, as a first simple canonical model for public opinion dynamics \cite{sirbu2017opinion,stauffer2009} in presence of a binary choice. In this context, the state of a spin describes the current opinion of an individual, the external magnetic field captures the exposure to biased information and/or one-sided marketing/campaigning, and the couplings between neighboring spins model the effect of peer interactions on personal opinions. 

The basic Ising model can be augmented to have more than two opinions and possibly asymmetric interactions between them, like in~\cite{vazquez2003constrained} where the authors consider an Ising-like model with three opinions but where the two most extreme opinions do not interact with each other. Since we are mostly interested in the interplay between opinion dynamics and network topology, in this paper we focus on the simpler case of a binary opinion. The voter model is another Ising-like model to study the evolution of binary opinions which features a different (and possibly irreversible) majority update rule, see e.g.~\cite{balankin2017ising,de1993nonequilibrium,de1992isotropic}. For a more broad review of mathematical and physical opinion dynamics models, we refer the interested reader to~\cite{xie2016review}.

In Ising-like binary opinion models, the temperature of the system approximates all the more or less random events which may influence individuals' opinions but are not explicitly accounted for in the model, cf.~\cite{stauffer2009}. In this paper, we study the Ising model in the low-temperature limit, which is instrumental to describe a situation where peer interactions and external factors have a strong influence on everyone's opinion. The low temperature favors homogeneous opinion patterns in which there are fewer individuals that disagree with the peers they interact with, which at a macroscopic level means that opinions become very rigid and hard to change, e.g., on a very polarizing issue.

It is clear that assuming the underlying structure is a lattice or the complete graph is not ideal when modeling public opinion dynamics, since individuals have very heterogeneous social networks and interaction patterns. In particular, it is reasonable to assume that each individual has only a finite number of interactions and that he/she would tend to align more with the opinion of individuals in the community we belong to rather than that of complete strangers. Aiming to understand the role of the community structure in opinion dynamics, in this paper we consider a very heterogeneous family of networks with very dense communities and very weak interactions between these communities. Various opinion dynamics models have been studied on networks with a community structure, e.g., \cite{lu2017impact,si2009opinion}, but mostly by means of numerical simulations, while in this paper we focus on rigorous mathematical results. 

By choosing a specific network structure, one may model also unilateral influences and/or negative influences. For this reason, the Ising model for binary opinion dynamics has been studied on signed networks~\cite{li2019binary} and directed networks~\cite{frahm2019ising}. Being primarily interested in the role of communities on opinion dynamics, in this paper we restrict ourselves to the nonsigned and undirected networks.

The structure of the network heavily influences both static (i.e., the configurations' energy) and dynamic properties (the likelihood of the system's trajectories) of the Ising model. In this setting, it is of interest to study the metastability or tunneling phenomena that the opinion dynamic model may exhibit. For instance, in presence of a positive external magnetic field, the metastable state of the system describes the diffusion of a second very rigid opinion which is not aligned with the mainstream one.

Informally, the metastable configurations are those in which the system persists for a long time before reaching one of the stable configurations, i.e., those minimizing the system's energy. In the context of the clustered network that we consider in this paper, the set of metastable states heavily depends on the relative strength of the interactions between the network communities and that of the external magnetic field. In absence of an external magnetic field, the two opinions are equally likely and the two homogeneous opinion patterns are both stable states. In this case, it is still interesting to study how, starting with all individuals agreeing on one opinion, the whole network can transition to the opposite opinion, how long this will take and what are the most likely trajectories of this process.

In this paper, we thus analyze the Ising model on a specific family of clustered networks, by identifying the set of metastable and stable states and by estimating the asymptotic behavior of the transition time between them in the low-temperature limit.

In order to study the metastability phenomenon, we adopt the statistical mechanics framework known as {\it pathwise approach}, which is the first dynamical approach to these phenomena initiated in \cite{Cassandro1984}, developed in \cite{Olivieri1995,Olivieri1996}, and later summarized in the monograph \cite{Olivieri2005}. This approach relies on a detailed knowledge of the energy landscape and large-deviation estimates to give a quantitative answer to the dynamical properties of the system during the transition from metastable to stable states. In particular, using this approach is possible to provide a convergence in probability, expectation, and law of the transition time, together with the description of the critical configurations and the tube of typical trajectories followed by the system. The pathwise approach has been later extended in~\cite{Nardi2015} to analyze the tunneling phenomenon, that is the asymptotic behavior of a system with more than one stable state and, in particular, its transition from a stable state to another stable state. A modern version of the pathwise approach can be found in \cite{Cirillo2013,Cirillo2015,Manzo2004,Nardi2015}.

Another approach is the {\it potential-theoretic approach} initiated in \cite{Bovier2002}. This method focuses on a precise analysis of hitting times of metastable sets with the help of potential theory. A crucial role in this approach is played by the so-called capacities, which can be estimated by exploiting variational principles, and might lead to sharper estimates for the transition time from metastable states to stable states. We refer to the monograph \cite{Bovier2015} for a detailed discussion of this approach and its applications to specific models. The potential-theoretic approach, however, is not always equivalent to the pathwise approach because they intrinsically rely on different definitions of metastable states. The situation is particularly delicate for evolutions of infinite-volume systems, irreversible systems and degenerate systems, as discussed in \cite{Cirillo2013,Cirillo2015,Cirillo2017,Bet2021PCA}. More recent approaches are developed in \cite{Beltrn2010,Beltrn2014,Bianchi2016,Bianchi2020}.

The pathwise approach was used to study the low-temperature behavior of finite-volume models with single-spin-flip Glauber dynamics, e.g.~\cite{Apollonio2022,https://doi.org/10.48550/arxiv.2208.11869,Bet2021,https://doi.org/10.48550/arxiv.2108.04011,doi:10.1063/5.0099480,Nardi2019,Zocca2018,Zocca2018bis}
, with Kawasaki dynamics, e.g.~\cite{https://doi.org/10.48550/arxiv.2208.13573,Baldassarri2021,Baldassarri2022weak,Baldassarri2022strong,denHollander2000,
Nardi2005}, and with parallel dynamics, e.g.~\cite{
Cirillo2008,Cirillo2008bis,Cirillo2022}. The potential theoretic approach was applied to the finite-volume Ising models at low temperature in~\cite{Bovier2005,
denHollander2011,denHollander2012,Nardi2012} for instance. 


The rest of the article is organized as follows. In \cref{sec:results}, we formally introduce the Ising model and the clustered network structure we consider in this paper and outline the main results for the transition time and the critical configurations for the dynamics. In \cref{sec:tools}, we give some definitions and present some preliminary results concerning the energy of the configurations. In \cref{sec:proof0}, we prove the main results in absence of a external magnetic field, whereas \cref{sec:proofh} is devoted to the proofs for the case of a positive external magnetic field. Finally, in \cref{sec:conclusions} we draw our conclusions and outline some future research directions.

\section{Model descriptions and main results}
\label{sec:results}

\subsection{Ising model on clustered graphs}
\label{sub:model}
In this paper, we are primarily interested in understanding the interplay between opinion dynamics and the community structure of the underlying network. Aiming to derive closed-form results, we choose a specific family of simple yet prototypical clustered networks. More specifically, we consider the Ising model on a graph $G$ consisting of $k$ clusters of equal size, which are locally complete graphs, and such that each node is connected to a single node in each of the other clusters. With this choice, we obtain a network with very dense communities which are only sparsely connected to each other. 

More specifically, for every $k\geq 2$ and every $n\geq 2$ we consider an undirected graph $G = \mathcal{G}(k,n)$ consisting of $k$ clusters, each of which is a complete subgraph of size $n$, in which we further connect each node, $i=1,\dots,n$ also to its $k-1$ ``twins'' in the other $k-1$ clusters (those whose labels have the same reminder modulo $n$), hence obtaining a regular graph where each node has degree $n+k-2$.

The vertex set of $\mathcal{G}(k,n)$ is $V = \bigcup_{i=1}^k V^{(i)}$ where $V^{(i)} := \{ n \cdot (i-1)+1,\dots,n \cdot i\}$ are the nodes in the $i$-th cluster. The edge set of $\mathcal{G}(k,n)$ is $E = E_{\mathrm{int}} \cup E_{\mathrm{cross}}$, where $E_{\mathrm{int}}=\bigcup_{i=1}^k E_{\mathrm{int}}^{(i)}$ is the collection of \textit{internal edges}, e.g., edges inside a cluster, and $E_{\mathrm{cross}}$ that of the edges across clusters, to which we refer as \textit{cross-edges}. The graph $\mathcal{G}(k,n)$ then has $\frac{1}{2} k n (n+k-2)$ edges, $n \binom{k}{2}$ of which are cross-edges and $\binom{n}{2}$ inside each cluster. \cref{fig:manifold7} depicts an instance of~$\mathcal{G}(2,7)$.

To each site $i\in V$ we associate a spin variable $\sigma(i)\in\{-1,+1\}$. We interpret $\sigma(i)=+1$ (resp.\ $\sigma(i)=-1$) as indicating that the spin at site $i$ is pointing upwards (resp.\ downwards). On the configuration space $\cX=\{-1,+1\}^{V}$, we consider the following \textit{Hamiltonian} or \textit{energy function}
\begin{equation}\label{eq:Ham}
    H(\sigma) := - \sum_{(i,j) \in E_{\mathrm{int}}} \sigma_i \sigma_j - \epsilon \sum_{(i,j) \in E_{\mathrm{cross}}} \sigma_i \sigma_j
    -h\sum_{i\in V}\sigma_i, 
\end{equation}
where we assume the strength of interaction across clusters is parametrized by a scalar $\epsilon \in [-1,1]$, while is equal to $1$ along all the other internal edges, and $h \in [0,1]$ is the external magnetic field. 

In the context of the binary opinion dynamics, a nonzero external magnetic field with $h>0$ is instrumental to describe a biased external influence, e.g., the exposure to biased information, or one-sided marketing/campaigning. Furthermore, it is reasonable to assume that the opinions of individuals that belong to a different community have less influence over us. For this reason, the interactions across different network clusters are assumed to be weaker than those inside each cluster, since their strength is equal to $|\epsilon| \leq 1$. Moreover, by taking negative values for $\epsilon$, we can model situations in which individuals tend to disagree with individuals from other communities.

We assume the systems evolves on $\cX$ according the single-flip Metropolis dynamics $(X_t)_{t\in\mathbb{N}}$ induced by the energy $H$ and parametrized by the inverse temperature $\beta >0$, whose transition probabilities are given by
\begin{equation}\label{eq:glauber}
    P(\sigma,\eta)=q(\sigma,\eta)e^{-\beta[H(\eta)-H(\sigma]_{+}}, \quad \text{for all } \sigma\neq\eta,
\end{equation}
where $[\cdot]_{+}$ denotes the positive part. The function $q(\sigma,\eta)$ is a connectivity matrix independent of $\beta$ that describess the possible transitions in $\cX$ and is defined for every $\sigma\neq\eta$ as
\begin{equation}\label{eq:conn}
    q(\sigma,\eta)=
    \begin{cases}
    \frac{1}{|V|} &\text{ if $\exists \ v\in V$ such that $\sigma^{(v)}=\eta$}, \\
    0 &\text{ otherwise},
    \end{cases}
\end{equation}
where $\sigma^{(v)} \in \cX$ is the configuration almost identical to $\sigma$ where only the spin of node $v$ has been flipped, i.e.,
\begin{equation}
    \sigma^{(v)}_i=
    \begin{cases}
    \sigma_i &\text{ if $i\neq v$}, \\
    -\sigma_i &\text{ if $i=v$}.
    \end{cases}
\end{equation}

Thus, the \textit{energy landscape} we consider is a tuple $(\mathcal{X}, \mathcal{Q}, H, \Delta)$ where $\cX$ is the state space, $\cQ \subset \cX \times \cX$ is the connectivity relation defined in \eqref{eq:conn}, $H$ is the energy function defined in \eqref{eq:Ham}, and the \emph{cost function} $\Delta: \cQ \to \mathbb{R}^+$ is defined as $\Delta(x, y) := [H (y)-H (x)]_+$. 
Note that the chosen energy landscape $(\cX, \cQ, H, \Delta)$ is \emph{reversible} with respect to the Gibbs measure
\[
    \mu(\sigma) = Z^{-1} \exp(-\beta H(\sigma)),
\]
where $Z=\sum_{\sigma \in \cX} H(\sigma)$ is the normalizing constant.

In the rest of the paper, we focus solely on the case of $k=2$ clusters, hence focusing on the family of networks $\mathcal{G}(2,n)$. The reason behind this choice is twofold: firstly, the case $k=2$ already exhibits a very diverse and rich behavior, and, secondly, the more general case with $k>2$ clusters is not conceptually harder to tackle, but simply heavier in terms of notation and terminology. 

Having a network with only $k=2$ clusters $V^{(1)}$ and $V^{(2)}$ allows for a very compact notation for spin configurations that are equivalent modulo relabelling of the nodes. For a configuration $\sigma \in \cX$ and $i=1,2$, let $V^{(i)}_+(\sigma)$ the subset of nodes in cluster $i$ whose spin is equal $+1$ in $\sigma$ and $E_+(\sigma)$ the subset of edges connecting $V^{(1)}_+(\sigma)$ and $V^{(2)}_+(\sigma)$. For $0 \leq p_1, p_2 \leq n$ and $0 \leq a \leq n$, we define the subset $C(p_1,p_2,a) \subset\mathcal X$ as
\[
    C(p_1,p_2,a) := \left \{ \sigma \in \cX ~:~ |V^{(1)}_+(\sigma)| = p_1, \ |V^{(2)}_+(\sigma)| = p_2, \text{ and } |E_+(\sigma)| = a \right \}.
\]
In words, $C(p_1,p_2,a)$ is the collection configurations $\sigma$ on $\mathcal{G}(2,n)$, such that
\begin{itemize}
    \item $\sigma$ has $0 \leq p_1 \leq n$ spins $+1$ on the first cluster, 
    \item $\sigma$ has $0 \leq p_2 \leq n$ spins $+1$ on the second cluster,
    \item $\sigma$ has $a$ of agreeing cross-edges between spins $+1$ in the first cluster and spins $+1$ on the second cluster. 
\end{itemize}
Note that the number $a$ of agreeing edges given $n, p_1, p_2$ must satisfy the following inequality
\[
    \max\{0,p_1+p_2-n\} \leq a \leq \min\{p_1,p_2\},
\]
since there cannot be a negative amount of edges between any pair of sub-clusters. \Cref{fig:manifold7} shows an example of a configuration in $C(4,3,3)$ on the network $\mathcal{G}(2,7)$.

\begin{figure}[!ht]
    \centering
    \raisebox{-0.5\height}{\includegraphics[width=0.5\textwidth]{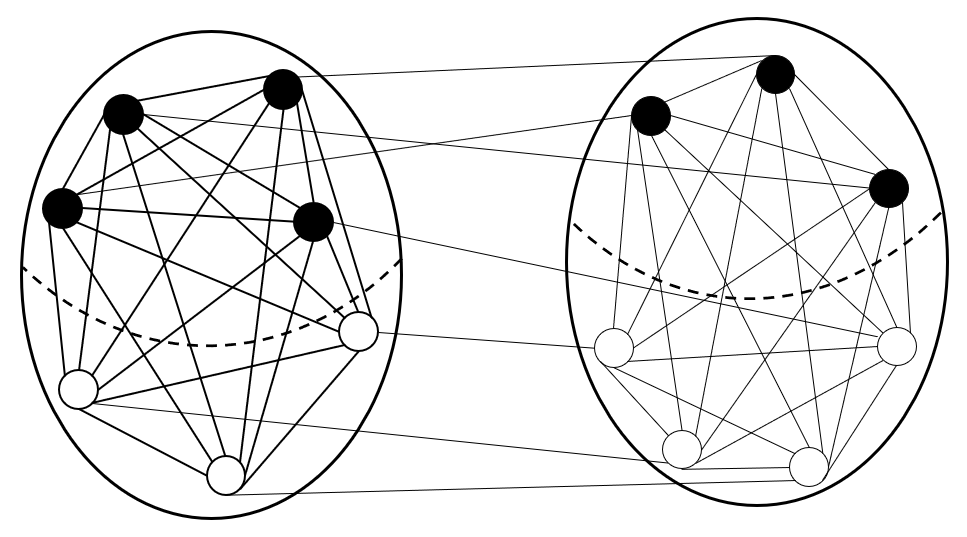}}
    \caption{Example of a configuration $\sigma \in C(4,3,3)$ on the network $\mathcal{G}(2,7)$ the with color-coded spins (black for $+1$ and white for $-1$). The first cluster has $p_1=4$ spin $+1$, has $p_2=3$ spin $+1$, and there are $a=3$ agreeing edges between plus spins.}
    \label{fig:manifold7}
\end{figure}
\FloatBarrier

We further denote by $\ppone, \mmone$ the two homogeneous configurations on $\mathcal{G}(2,n)$ consisting of all $+1$ spins and all $-1$ spins, see \cref{fig:uniform}. We refer to the configurations which are not globally homogeneous but are locally uniform inside each cluster as \textit{mixed configurations} and denote them as $\pmone, \mpone$. Clearly, there are only $2$ of them on $\mathcal{G}(2,n)$, see \cref{fig:mix}.

\begin{figure}[!ht]
    \centering
    \raisebox{-0.5\height}{\includegraphics[width=0.35\textwidth]{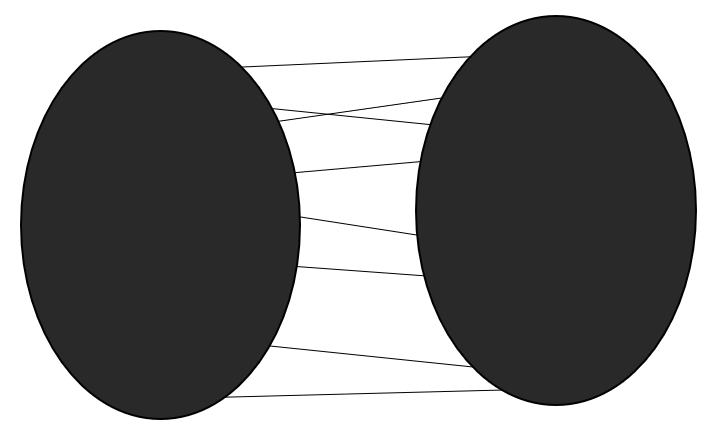}}
    \qquad \qquad
    \raisebox{-0.5\height}{\includegraphics[width=0.35\textwidth]{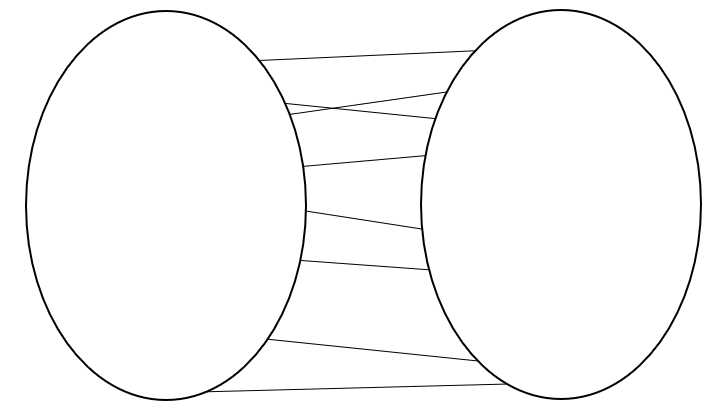}}
    \caption{The two uniform configurations $\ppone$ and $\mmone$, where we represent in white (resp.\ black) the minus (resp.\ plus) spins.}
    \label{fig:uniform}
\end{figure}

\begin{figure}[!ht]
    \centering
    \raisebox{-0.5\height}{\includegraphics[width=0.35\textwidth]{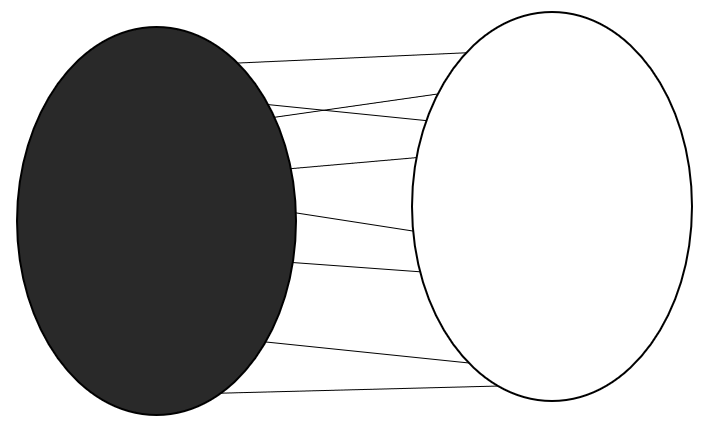}}
    \qquad \quad
    \raisebox{-0.5\height}{\includegraphics[width=0.35\textwidth]{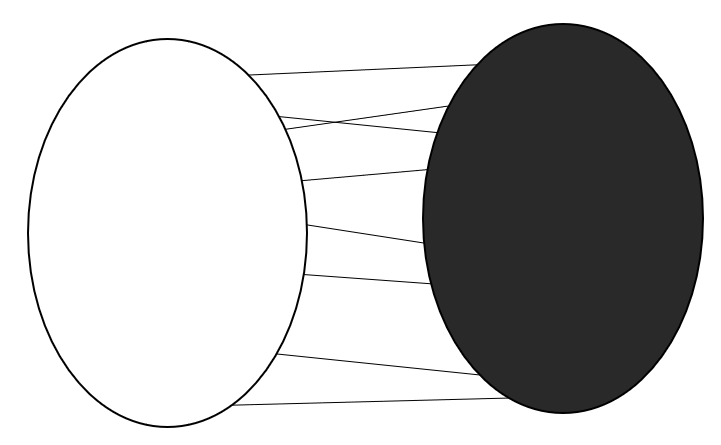}}
    \caption{The two mixed configurations $\pmone$ and $\mpone$, where we represent in white (resp.\ black) the minus (resp.\ plus) spins.}
    \label{fig:mix}
\end{figure}
\FloatBarrier

\subsection{Asymptotic behavior of the model in the low-temperature regime}
For all values of the external magnetic field $h\in [0,1]$, the considered Ising model exhibits a metastable behavior. In this section, we state our main results, which concern the analysis of the transition either from a metastable to a stable state, or between two stable states, and the description of the corresponding critical configurations. Even if refer the reader to \cref{modinddef} for the precise definitions of a stable state and a metastable state, we want to provide some intuition for them before stating the main results. A stable state is easily defined as a configuration that is a global minimum of the energy $H$. On the other hand, metastable states cannot be identified only by looking at the energy $H$, as they are intrinsically defined by the evolution of the system as those configurations in which the system resides the longest before arriving in one of the stable states. In terms of the energy landscape, the metastable states are those leaving from which the dynamics has to overcome the largest energy barrier.

In \cref{sec:h=0} we state our main results for the case $h=0$, while in \cref{sec:h>0} those for the case $h>0$. Our results concern the asymptotic behavior of the transition times between metastable and stable configurations in the limit as $\beta\rightarrow\infty$, as well as the identification of the so-called \textit{gate} of critical configurations, which represents a set of configurations that will be crossed with very high probability along these transitions (cf.~point 4 in \cref{modinddef} for the precise definition). 

\subsubsection{Case \texorpdfstring{$h=0$}{h=0}}\label{sec:h=0}
In this subsection we focus on the case $h=0$, namely there is no external magnetic field. The first result we provide is the identification of metastable and stable states, which is the subject of the following theorem.

\begin{theorem}[Stable and metastable states]\label{thm:metastabstates}
Let $(\cX, Q, H, \Delta)$ be the energy landscape corresponding to the Ising model on $\mathcal{G}(2,n)$. Then, the lowest possible energy is equal to
\begin{equation} \label{eq:Hmin}
    \min_{\sigma \in \cX} H(\sigma) = -n^2 + n -|\epsilon| n.
\end{equation}
The set of stable states is
\[
    \cX_s=
    \begin{cases}
       \{\ppone, \mmone\} & \text{ if } \epsilon>0,\\
       \{\ppone, \mmone,\pmone, \mpone\} & \text{ if } \epsilon = 0,\\
       \{\pmone, \mpone\} & \text{ if } \epsilon<0,
    \end{cases}
\]
and the set of metastable states is
\[
    \cX_m=
    \begin{cases}
       \{\pmone, \mpone\} & \text{ if } \epsilon>0,\\
       \{\ppone, \mmone\} & \text{ if } \epsilon<0.
    \end{cases}
\]
\end{theorem}

The next theorem investigates the asymptotic behavior as $\beta\rightarrow\infty$ of the tunneling time for the system started at the stable state $s_1$ to reach for the first time the other stable state $s_2$, which we denote by $\tau^{s_1}_{s_2}$. See \eqref{tempo} for the precise definition. In order to state the theorem, we need to define 
\begin{equation}\label{eq:gamma0s}
\Gamma^0_s :=
\begin{cases}
     \frac{n^2}{2} + |\epsilon| n &\text{ if $n$ is even}, \\
     \frac{n^2-1}{2} + |\epsilon| (n + 1) &\text{ if $n$ is odd},
    \end{cases}
    \end{equation}
that represents the maximal value of the energy barrier between two stable states.

\begin{theorem}[Asymptotic behavior of the tunneling time]\label{thm:tunneling}
For any $\delta>0$ and for any $s_1,s_2\in\cX_s$, the following statements hold
\begin{itemize}
\item[(i)] $\displaystyle\lim_{\beta\rightarrow\infty}\mathbb{P}(e^{\beta(\Gamma^0_s-\delta)}<\tau^{s_1}_{s_2}<e^{\beta(\Gamma^0_s+\delta)})=1$;
\item[(ii)] $\displaystyle\lim_{\beta\rightarrow\infty}\frac{1}{\beta}\log\mathbb{E}\tau^{s_1}_{s_2}=\Gamma^0_s$;
\item[(iii)] $\displaystyle\frac{\tau^{s_1}_{s_2}}{\mathbb{E}\tau^{s_1}_{s_2}}\overset{d}{\rightarrow}{\rm Exp}(1)$ as $\beta\rightarrow\infty$;
\item[(iv)] there exist two constants $0<c_1\leq c_2<\infty$ independent of $\beta$ such that for every $\beta>0$
\begin{equation}
c_1 e^{-\beta\Gamma^0_s}\leq\rho_\beta\leq c_2e^{-\beta\Gamma^0_s},
\end{equation}
where $\rho_\beta$ is the spectral gap of the Markov process.
\end{itemize}
\end{theorem}

\begin{remark}
We note that \cref{thm:tunneling}(iv) implies that 
\begin{equation}
\displaystyle\lim_{\beta\rightarrow\infty}\frac{1}{\beta}\log t_{mix}(\gamma)=\Gamma^0_s=\lim_{\beta\rightarrow\infty}-\frac{1}{\beta}\log \rho_\beta,
\end{equation}
where $t_{mix}(\gamma)$ is the mixing time of the Markov process, which quantifies how long it takes the empirical
distribution of the process to get close to the stationary distribution (see \cref{modinddef} point 1 for the precise definition).
\end{remark}

The last result of this section concerns the description of a gate for the transition between the stable states $s_1$ and $s_2$. To this end, if $n$ is odd, we define
\begin{equation}
     C^*_{odd}:=
     \begin{cases}
     C\left(\frac{n+1}{2},0,0\right)\cup C\left(0,\frac{n+1}{2},0\right) \cup C\left(n,\frac{n-1}{2},\frac{n-1}{2}\right) \cup C\left(\frac{n-1}{2},n,\frac{n-1}{2}\right) & \text{ if } \epsilon\geq0, \\ 
     C\left(\frac{n-1}{2},0,0\right)\cup C\left(0,\frac{n-1}{2},0\right) \cup C\left(n,\frac{n+1}{2},\frac{n+1}{2}\right) \cup C\left(\frac{n+1}{2},n,\frac{n+1}{2}\right) & \text{ if } \epsilon<0,
     \end{cases}
\end{equation}
\noindent otherwise if $n$ is even, we define
\begin{equation}
   C^*_{even}:= 
   C\left(\frac{n}{2},0,0\right)\cup C\left(0,\frac{n}{2},0\right) \cup C\left(n,\frac{n}{2},\frac{n}{2}\right) \cup C\left(\frac{n}{2},n,\frac{n}{2}\right).
\end{equation}

\begin{theorem}[Gate for the tunneling transition]\label{thm:gate}
If $n$ is even (resp.\ odd), the set $C^*_{even}$ (resp.\ $C^*_{odd}$) is a gate for the transition from $s_1$ to $s_2$ for any $s_1,s_2\in\cX_s$.
\end{theorem}

\subsubsection{Case \texorpdfstring{$h>0$}{h>0}}\label{sec:h>0}
In this subsection, we focus on the case $h>0$, which describes the situation in which there is a positive external magnetic field that favors plus spins. Moreover, we assume that $0<h\leq1$ in order to avoid the energetical contribution of the external magnetic field prevails over the binding energies associated with internal edges. As it will be clear later, the dynamical behavior of the system is different in the two cases $0<h\leq|\epsilon|\leq1$ and $0\leq|\epsilon|<h\leq1$, especially when $\epsilon<0$. Indeed, this corresponds to a different ``importance" given to cross-edges and external magnetic field. The first result we provide is the identification of metastable and stable states, which is the subject of the following theorem.
\begin{theorem}[Stable and metastable states]\label{thm:metstatesh}
Let $(\cX, Q, H, \Delta)$ be the energy landscape corresponding to the Ising model on $\mathcal{G}(2,n)$. Then, the lowest possible energy is equal to
\begin{equation}\label{eq:Hmin2}
    \min_{\sigma \in \cX} H(\sigma) =
    \begin{cases}
    -n^2 + n -\epsilon n -2hn & \text{ if } 0\leq\epsilon\leq1 \text{ or } 0<-\epsilon<h\leq1, \\  
     -n^2 + n +\epsilon n  & \text{ if } 0<h\leq-\epsilon\leq1.
    \end{cases}
\end{equation}

The set of stable states is
\[
    \cX_s=
    \begin{cases}
       \{\ppone\} & \text{ if } 0\leq\epsilon\leq1 \text{ or } 0<-\epsilon<h\leq1,\\
       \{\ppone, \pmone, \mpone\} & \text{ if } h=-\epsilon, \\
       \{\pmone, \mpone\} & \text{ if } 0<h<-\epsilon\leq1,
    \end{cases}
\]
and the set of metastable states is
\[
    \cX_m=
    \begin{cases}
    \{\mmone\} &\text{ if } 0\leq\epsilon\leq1 \text{ or } h=-\epsilon, \\
    \{\pmone,\mpone\} & \text{ if } 0<-\epsilon<h\leq1, \\
    \{\ppone\} & \text{ if } 0<h<-\epsilon\leq1.
    \end{cases}
\]
\end{theorem}

The next theorems investigate the asymptotic behavior as $\beta\rightarrow\infty$ of the tunneling time (resp.\ transition time to the stable state) for the system started at the stable state $s_1$ (resp.\ metastable state $m$) to reach for the first time the other stable state $s_2$ (resp.\ the stable state $s$) if $0<h<-\epsilon\leq1$ (resp.\ if $0\leq\epsilon\leq1$ or $0<-\epsilon<h\leq1$). See \eqref{tempo} for the precise definition. In order to state the theorems, we need to define:
\begin{align}
& \Gamma^1_m := \label{eq:gamma1m}
\begin{cases}
    \frac{n^2}{2}+n(\epsilon-h) &\text{ if $n$ is even}, \\
    \frac{n^2-1}{2}+(n+1)(\epsilon-h) &\text{ if $n$ is odd and } 0<h\leq\epsilon\leq1, \\
    \frac{n^2-1}{2}+(n-1)(\epsilon-h) &\text{ if $n$ is odd and } 0\leq\epsilon<h\leq1,
    \end{cases} 
\\
& \Gamma^2_m := \label{eq:gamma2m}
\begin{cases}
    \frac{n^2}{2}-n(\epsilon+h) &\text{ if $n$ is even}, \\
    \frac{n^2-1}{2}-(n-1)(\epsilon+h) &\text{ if $n$ is odd},
    \end{cases} 
\\
& \Gamma^h_s := \label{eq:gammahs}
\begin{cases}
     \frac{n^2}{2}+n(h-\epsilon)&\text{ if $n$ is even and }0<h-\epsilon<1, \\
     \frac{n^2-4}{2}+(n+2)(h-\epsilon)&\text{ if $n$ is even and }1\leq h-\epsilon<2, \\
     \frac{n^2-1}{2}+(n+1)(h-\epsilon)&\text{ if $n$ is odd}.
    \end{cases}
\end{align}
that represent the maximal values of the energy barrier between the set of metastable states to the set of stable states or between two stable states.


\begin{theorem}[Asymptotic behavior of the tunneling time]\label{thm:tunnelingtimeh}
If $0<h<-\epsilon\leq1$, for any $\delta>0$ and for any $s_1,s_2\in\cX_s$, the following statements hold
\begin{itemize}
\item[(i)] $\displaystyle\lim_{\beta\rightarrow\infty}\mathbb{P}(e^{\beta(\Gamma^h_s-\delta)}<\tau^{s_1}_{s_2}<e^{\beta(\Gamma^h_s+\delta)})=1$;
\item[(ii)] $\displaystyle\lim_{\beta\rightarrow\infty}\frac{1}{\beta}\log\mathbb{E}\tau^{s_1}_{s_2}=\Gamma^h_s$;
\item[(iii)] $\displaystyle\frac{\tau^{s_1}_{s_2}}{\mathbb{E}\tau^{s_1}_{s_2}}\overset{d}{\rightarrow}{\rm Exp}(1)$ as $\beta\rightarrow\infty$;
\item[(iv)] there exist two constants $0<c_1\leq c_2<\infty$ independent of $\beta$ such that for every $\beta>0$
\begin{equation}
c_1 e^{-\beta\Gamma^h_s}\leq\rho_{\beta}\leq c_2e^{-\beta\Gamma^h_s},
\end{equation}
where $\rho_{\beta}$ is the spectral gap of the Markov process. 
\end{itemize}
\end{theorem}
If $0 \leq \epsilon \leq 1$ we set $\Gamma^*_m=\Gamma^1_m$, whereas if $0<-\epsilon<h\leq 1$ we set $\Gamma^*_m=\Gamma^2_m$.
\begin{theorem}[Asymptotic behavior of the transition time]\label{thm:transitiontime}
If $0\leq\epsilon\leq1$ or $0<-\epsilon<h\leq1$, for any $\delta>0$, for $m\in\cX_m$ and $s\in\cX_s$, the following statements hold
\begin{itemize}
\item[(i)] $\displaystyle\lim_{\beta\rightarrow\infty}\mathbb{P}(e^{\beta(\Gamma^*_m-\delta)}<\tau^{m}_{s}<e^{\beta(\Gamma^*_m+\delta)})=1$;
\item[(ii)] $\displaystyle\lim_{\beta\rightarrow\infty}\frac{1}{\beta}\log\mathbb{E}\tau^{m}_{s}=\Gamma^*_m$;
\item[(iii)] $\displaystyle\frac{\tau^{m}_{s}}{\mathbb{E}\tau^{m}_{s}}\overset{d}{\rightarrow}{\rm Exp}(1)$ as $\beta\rightarrow\infty$;
\item[(iv)] there exist two constants $0<c_1\leq c_2<\infty$ independent of $\beta$ such that for every $\beta>0$
\begin{equation}
c_1 e^{-\beta\Gamma^*_m}\leq\rho_{\beta}\leq c_2e^{-\beta\Gamma^*_m},
\end{equation}
where $\rho_\beta$ is the spectral gap of the Markov process.
\end{itemize}
\end{theorem}

\begin{remark}
We note that \cref{thm:tunnelingtimeh}(iv) implies that 
\begin{equation}
\displaystyle\lim_{\beta\rightarrow\infty}\frac{1}{\beta}\log t_{mix}(\gamma)=\Gamma^h_s=\lim_{\beta\rightarrow\infty}-\frac{1}{\beta}\log \rho_\beta,
\end{equation}
where $t_{mix}(\gamma)$ is the mixing time of the Markov process (see \cref{modinddef} point 1 for the precise definition). Analogously, a similar result can be also derived for $\Gamma^*_m$ from \cref{thm:transitiontime}(iv).
\end{remark}

The last main result of this section concerns the description of a gate for the transition between the stable states $s_1$ and $s_2$ (resp.\ between the metastable state $m$ and the stable state $s$) if $0<h<-\epsilon\leq1$ (resp.\ if $0\leq\epsilon\leq1$ or $0<-\epsilon<h\leq1$). To this end, we need the following definitions. \\
If $0\leq\epsilon\leq1$, we define
\begin{equation}
     C^*_{1}:=
     \begin{cases}
     C\left(\frac{n+1}{2},0,0\right)\cup C\left(0,\frac{n+1}{2},0\right) &\text{if $n$ is odd and } 0<h\leq\epsilon\leq1, \\ 
     C\left(\frac{n-1}{2},0,0\right)\cup C\left(0,\frac{n-1}{2},0\right) & \text{if $n$ is odd and } 0\leq\epsilon<h\leq1, \\
     C\left(\frac{n}{2},0,0\right) \cup C\left(0,\frac{n}{2},0\right) & \text{if $n$ is even}.
     \end{cases}
\end{equation}

\noindent 
If $0<-\epsilon<h\leq1$, we define

\begin{equation}
   C^*_{2}:= 
   \begin{cases}
   C\left(n,\frac{n-1}{2},\frac{n-1}{2}\right) &\text{if $n$ is odd}, \\
   C\left(n,\frac{n}{2},\frac{n}{2}\right) &\text{if $n$ is even}.
   \end{cases}
\end{equation}

\noindent 
If $0<h<-\epsilon\leq1$, we define

\begin{equation}
   C^*_{3}:= 
   \begin{cases}
   C\left(\frac{n-1}{2},0,0\right)\cup C\left(0,\frac{n-1}{2},0\right) &\text{if $n$ is odd}, \\
   C\left(\frac{n}{2},0,0\right)\cup C\left(0,\frac{n}{2},0\right) &\text{if $n$ is even and } 0<h-\epsilon<1, \\
   C\left(\frac{n-2}{2},0,0\right)\cup C\left(0,\frac{n-2}{2},0\right) &\text{if $n$ is even and } 1\leq h-\epsilon<2.
   \end{cases}
\end{equation}

\begin{theorem}[Gate for the transition]\label{thm:gateh}
If $0\leq\epsilon\leq1$ (resp.\ $0<-\epsilon<h\leq1$), the set $C_1^*$ (resp.\ $C^*_2$) is a gate for the transition from the metastable state $m$ to the stable state $s$. If $0<h<-\epsilon\leq1$, the set $C^*_3$ is a gate for the transition from $s_1$ to $s_2$ for any $s_1,s_2\in\{\pmone,\mpone\}$.
\end{theorem}

\section{Main tools and preliminaries}
\label{sec:tools}
\subsection{Model-independent definitions and notations}
\label{modinddef}
In this section, we provide some definitions and notations which will be useful throughout the paper. 

\medskip
\noindent
{\bf 1. Paths, hitting, and mixing times.}
\begin{itemize}

\item
A {\it path\/} $\omega$ is a sequence $\omega=(\omega_1,\dots,\omega_k)$, with
$k\in\mathbb{N}$, $\omega_i\in{\cal X}$ and $P(\omega_i,\omega_{i+1})>0$ for $i=1,\dots,k-1$.
We write $\omega\colon\;\eta\to\eta'$ to denote a path from $\eta$ to $\eta'$,
namely with $\omega_1=\eta,$ $\omega_k=\eta'$. 

\item
Given a non-empty set ${\cal A}\subset{\cal X}$ and a state $\sigma\in{\cal X}$, we define the {\it first-hitting time of} ${\cal A}$ with initial state $\sigma$ at time $t=0$ as
\begin{equation}\label{tempo}
\tau^{\sigma}_{\cal A}:=\min \{t\geq 0: X_t \in {\cal A} \, | X_0=\sigma\}.
\end{equation}
\item 
We define the {\it mixing time} as
\begin{equation}
t_{mix}(\gamma):=\min\{n\geq0: \max_{\sigma\in\cX}||P_n(\sigma,\cdot)-\mu(\cdot)||_{TV}\leq\gamma\},
\end{equation}

\noindent
where $||\nu-\nu'||_{TV}:=\frac{1}{2}\sum_{\sigma\in\cX}|\nu(\sigma)-\nu'(\sigma)|$ for any two probability distributions $\nu,\nu'$ on $\cX$. The {\it spectral gap} of the Markov chain is defined as
\begin{equation}
\rho_\beta:=1-a_\beta^{(2)},
\end{equation}

\noindent
where $1=a_\beta^{(1)}>a_\beta^{(2)}\geq...\geq a_\beta^{(|\cX|)}\geq-1$ are the eigenvalues of the matrix $(P(\sigma,\eta))_{\sigma,\eta\in\cX}$ defined in (\ref{eq:glauber}).
\end{itemize}
\noindent

\medskip
\noindent
{\bf 2. Communication height, stability level, stable and metastable states}
\begin{itemize}

\item 
The {\it communication height} between a pair $\eta$, $\eta'\in\cal X$ is
\begin{equation}
\Phi(\eta,\eta'):= \min_{\omega:\eta\rightarrow\eta'}\max_{\zeta\in\omega} H(\zeta).
\end{equation}


\item
We call
{\it stability level} of
a state $\zeta \in \cal X$
the energy barrier
\begin{equation}\label{stab}
V_{\zeta} :=
\Phi(\zeta,{\cal I}_{\zeta}) - H(\zeta),
\end{equation}

\noindent 
where $\cal I_{\zeta}$ is the set of states with energy below $H(\zeta)$:
\begin{equation}\label{iz} 
{\cal I}_{\zeta}:=\{\eta \in {\cal X}: H(\eta)<H(\zeta)\}.
\end{equation}

\noindent 
We set $V_\zeta:=\infty$ if $\cal I_\zeta$ is empty.

\item
The set of {\it stable states} is the set of the global minima of
the Hamiltonian and we denote it by $\cX_s$.
Moreover, for any $s_1,s_2\in\cX_s$, we set $\Gamma_s := \Phi(s_1,s_2)-H(s_1)$.

\item
The set of {\it metastable states} is given by
\begin{equation}\label{st.metast.} 
{\cal X}_m:=\{\eta\in{\cal X}:
V_{\eta}=\max_{\zeta\in{\cal X}\setminus {\cal X}_s}V_{\zeta}\}.
\end{equation}
\noindent
We denote by 
\begin{equation}
    \Gamma_m := \max_{\zeta \in \cX \setminus \cX_s} V_{\zeta}
\end{equation}
the maximum stability level, namely the stability level of the states in ${\cal X}_m$. We note that $\Gamma_m=\Phi(m,s)-H(m)$, where $m\in\cX_m$ and $s\in\cX_s$.
\end{itemize}

\medskip
\noindent
{\bf 4. Optimal paths, saddles, and gates}
\begin{itemize}

\item 
We denote by $(\eta\to\eta')_{opt} $ the {\it set of optimal paths\/} as the set of all
paths from $\eta$ to $\eta'$ realizing the min-max in $\cal X$, i.e.,
\begin{equation}\label{optpath}
(\eta\to\eta')_{opt}:=\{\omega:\eta\rightarrow\eta'\; \hbox{such that} \; \max_{\xi\in\omega} H(\xi)=  \Phi(\eta,\eta') \}.
\end{equation}

\item
The set of {\it minimal saddles\/} between
$\eta,\eta'\in\cal X$
is defined as
\begin{equation}\label{minsad}
{\cal S}(\eta,\eta'):= \{\zeta\in{\cal X}\colon\;\; \exists\omega\in (\eta\to\eta')_{opt},
\ \omega\ni\zeta \hbox{ such that } \max_{\xi\in\omega} H(\xi)= H(\zeta)\}.
\end{equation}

\item
Given a pair $\eta,\eta'\in\cal X$,
we say that $\cal W\equiv\cal W(\eta,\eta')$ is a {\it gate\/}
for the transition $\eta\to\eta'$ if $\cal W(\eta,\eta')\subseteq\cal S(\eta,\eta')$
and $\omega\cap\cal W\neq\emptyset$ for all $\omega\in (\eta\rightarrow\eta')_{opt}$. In words, a gate is a subset of $\cal S(\eta,\eta')$ that is visited by all optimal paths.

\end{itemize}

\subsection{Energetical properties of the configurations}
In this section, we provide some useful lemmas concerning the energetical properties of the configurations in $C(p_1,p_2,a)$, which will be used in the rest of the paper. Note that all the configurations in $C(p_1,p_2,a)$ are identical, {\it modulo permutation of the vertices in each cluster}, since:
\begin{itemize}
    \item the first cluster gets $n-p_1$ minus spins, 
    \item the second cluster gets $n-p_2$ minus spins,
    \item $p_1-a$ is the number of cross edges between plus spins in the first cluster and minus spins in the second cluster,
    \item $p_2-a$ is the number of cross edges between minus spins in the first cluster and plus spins in the second cluster,
    \item $n+a-p_1-p_2$ is the number of cross edges between minus spins in the first cluster and minus spins in the second cluster.
\end{itemize}



\begin{lemma}[Energy of the configurations]\label{lemma:enconf}
For any $\sigma\in C(p_1,p_2,a)$, it holds that
\begin{equation}\label{eq:Hp1p2a}
H(\sigma)=n - \epsilon n - 2 \left(p_1 - \frac{n}{2}\right)^2 - 2 \left(p_2 - \frac{n}{2}\right)^2 - 2 \epsilon (2a - p_1 - p_2) -2h(p_1+p_2-n).
\end{equation}
\end{lemma}
\begin{proof}
Let $\sigma\in C(p_1,p_2,a)$. Note that in the first cluster there are $\binom{p_1}{2}$ (resp.\ $\binom{n-p_1}{2}$) internal edges between plus (resp.\ minus) spins, whereas there are $p_1(n-p_1)$ internal edges between plus and minus spins. By symmetry, analogous relations can be derived for the second cluster. Moreover, there are $n+2a-p_1-p_2$ (resp.\ $p_1+p_2-2a$) cross edges between spins of the same (resp.\ different) type and $p_1+p_2$ plus spins in $\mathcal{G}(2,n)$. Thus, by using \eqref{eq:Ham} we deduce
\begin{equation*}
\begin{array}{ll}
    H(\sigma) & = -\displaystyle\frac{(n-p_1)(n-p_1-1)}{2} - \frac{p_1(p_1-1)}{2} - \frac{(n-p_2)(n-p_2-1)}{2} - \frac{p_2(p_2-1)}{2} \\ \\
    & \quad + p_1(n-p_1) + p_2(n-p_2) - \epsilon (n+4a-2p_1-2p_2) -2h(p_1+p_2-n) \\ \\
    &=n - \epsilon n - 2 \left(p_1 - \dfrac{n}{2}\right)^2 - 2 \left(p_2 - \dfrac{n}{2}\right)^2 - 2 \epsilon (2a - p_1 - p_2) -2h(p_1+p_2-n).
\end{array}
\end{equation*}
\end{proof}

From now on, we define up-flip (resp.\ down-flip) as the move consisting in flipping a minus (resp.\ plus) spin in a plus (resp.\ minus) spin. 
\begin{lemma}[Energy difference for an up-flip]\label{lmm:1}
Let $\sigma_1\in C(p_1,p_2,a_1)$ and let $\sigma_2\in C(p_1+i,p_2+j,a_2)$, with $i,j\in\{0,1\}$ such that $i\neq j$. Then,
    \begin{align}
        H(\sigma_2) - H(\sigma_1) = 
        \begin{cases}
            2(n-1-2p_1+\epsilon-h) & \text{if } i=1, \ p_1\leq n-1 \text{ and } a_2 = a_1,\\
            2(n-1-2p_1-\epsilon-h) & \text{if } i=1, \ p_1\leq n-1 \text{ and } a_2 = a_1+1, \\
            2(n-1-2p_2+\epsilon-h) & \text{if } j=1, \ p_2\leq n-1 \text{ and } a_2 = a_1, \\
            2(n-1-2p_2-\epsilon-h) & \text{if } j=1,  \ p_2\leq n-1 \text{ and } a_2 = a_1+1.
        \end{cases}
    \end{align}
\end{lemma}
\begin{proof}
In the case $i=1$ and $p_1\leq n-1$, by using \eqref{eq:Hp1p2a}, we directly get
\[
H(\sigma_2) - H(\sigma_1) = 2(n-1-2p_1+2\epsilon a_1-2\epsilon a_2+\epsilon-h)=
\begin{cases}
2(n-1-2p_1+\epsilon-h) & \text{if } a_2=a_1, \\
2(n-1-2p_1-\epsilon-h) & \text{if } a_2=a_1+1.
\end{cases}
\]
By symmetry, we get the claim also in the case $j=1$ and $p_2\leq n-1$.
\end{proof}

\begin{lemma}[Energy difference for a down-flip]\label{lmm:2}
Let $\sigma_1\in C(p_1,p_2,a_1)$, $\sigma_2\in C(p_1-i,p_2-j,a_2)$, with $i,j\in\{0,1\}$ such that $i\neq j$. Then,
    \begin{align}
        H(\sigma_2) - H(\sigma_1) = 
        \begin{cases}
            -2(n+1-2p_1+\epsilon-h) &\text{if } i=1, \ p_1\geq1 \text{ and } a_2 = a_1,\\
            -2(n+1-2p_1-\epsilon-h) &\text{if } i=1, \ p_1\geq1 \text{ and } a_2 = a_1-1, \\
            -2(n+1-2p_2+\epsilon-h) &\text{if } j=1, \ p_2\geq1 \text{ and } a_2 = a_1,\\
            -2(n+1-2p_2-\epsilon-h) &\text{if } j=1, \ p_2\geq1 \text{ and } a_2 = a_1-1.
        \end{cases}
    \end{align}
\end{lemma}
\begin{proof}
By proceeding as in the proof of \cref{lmm:2}, we get the claim.
\end{proof}

Since the configurations in $C(p_1,p_2,a)$ have all the same energy, see Lemma \ref{lemma:enconf}, with a slight abuse of notation in the rest of the paper we denote their energy value by $H(p_1,p_2,a)$.

\section{Proof of the main results: case \texorpdfstring{$h=0$}{h=0}}
\label{sec:proof0}

\subsection{Reference paths}

If $\epsilon\geq0$, we define a reference path $\bar\omega$ from $\mmone$ to $\ppone$, while if $\epsilon<0$ we define a path $\hat\omega$ from $\pmone$ to $\mpone$. In words, these paths are constructed in the following way. The path $\bar\omega$, which starts from $\mmone$, consists in flipping one by one the minus spins in one community until the path reaches either $\pmone$ or $\mpone$ and afterward the remaining minuses are flipped one by one until the path reaches $\ppone$ (see \cref{fig:pathbar}). The construction of the path $\hat\omega$ is made in a similar way (see \cref{fig:pathhat}).

\begin{figure}
\centering
    \includegraphics[width=\textwidth]{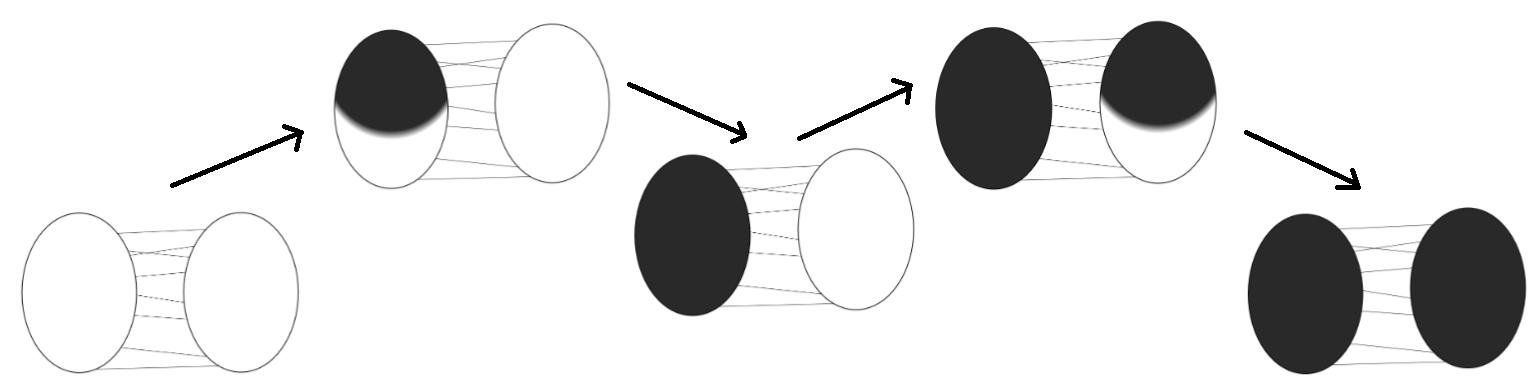}
    \caption{Here we depict the reference path $\bar\omega$ by representing the saddles, the metastable and stable states that it crosses, where we represent in white (resp.\ black) the minus (resp.\ plus) spins.}
    \label{fig:pathbar}
    \vskip 0.5cm
\end{figure}

\begin{definition}[Reference paths]\label{defbaromega} 
If $\epsilon\geq0$, we define $\bar\omega:\mmone\to\ppone$ as the path $(\bar\omega_k)_{k=0}^{2n}$ such that 
\begin{align}\label{eq:confomegabar}
    \bar\omega_k\in C(k,0,0)\ \text{ and }\ \bar\omega_{n+k}\in C(n,k,k), \quad \text{for any } k=0,\dots,n.
\end{align}

If $\epsilon<0$, we define $\hat\omega:\pmone\to\mpone$ as the path $(\hat\omega_k)_{k=0}^{2n}$ such that 
\begin{align}\label{eq:confomegabar2}
    \hat\omega_k\in C(n,k,k)\ \text{ and }\ \hat\omega_{n+k}\in C(n-k,n,n-k), \quad \text{for any } k=0,\dots,n.
\end{align}
\end{definition}

\begin{figure}[!htb]
\centering
    \includegraphics[width=\textwidth]{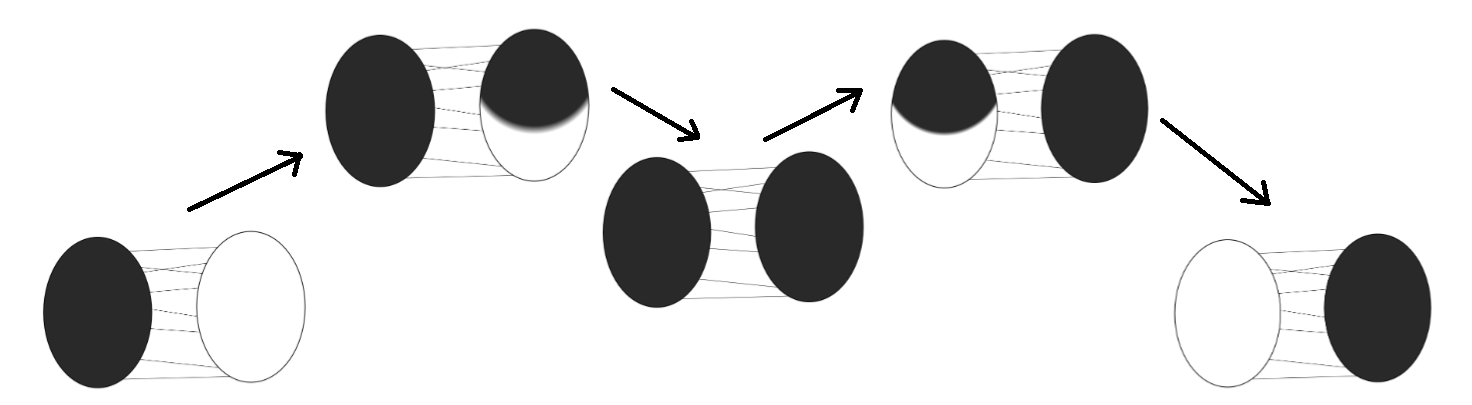}
    \caption{Here we depict the reference path $\hat\omega$ by representing the saddles, the metastable and stable states that it crosses, where we represent in white (resp.\ black) the minus (resp.\ plus) spins.}
    \label{fig:pathhat}
\end{figure}

\begin{lemma}[Maximal energy on the reference paths]\label{lemmamaxbar}
Let $\bar\omega:\mmone\to\ppone$ and $\hat\omega:\pmone\to\mpone$ be the paths given in \cref{defbaromega}. Then,
\begin{align}\label{eq:phibaromega}
    \Phi_{\bar\omega}=
    \begin{cases}
    H(\bar\omega_{\frac{n}{2}})=H(\bar\omega_{n+\frac{n}{2}})=n-\frac{n^2}{2} &\text{if $n$ is even}, \\
    H(\bar\omega_{\frac{n+1}{2}})=H(\bar\omega_{n+\frac{n-1}{2}})=n-\frac{n^2+1}{2}+\epsilon &\text{if $n$ is odd}, 
    \end{cases}
\end{align}
and
\begin{align}
    \Phi_{\hat\omega}=
    \begin{cases}
    H(\hat\omega_{\frac{n}{2}})=H(\hat\omega_{n+\frac{n}{2}})=n-\frac{n^2}{2}-\epsilon n &\text{if $n$ is even}, \\
    H(\hat\omega_{\frac{n+1}{2}})=H(\hat\omega_{n+\frac{n-1}{2}})=n-\frac{n^2+1}{2}-\epsilon &\text{if $n$ is odd}. 
    \end{cases}
\end{align}
\end{lemma}
\begin{proof}
Since $H(C(n-k,n,n-k))=H(C(k,0,0))$, it suffices to study the maxima of the energy along the path $\bar\omega$ connecting $\mmone$ and $\ppone$. From \eqref{eq:Hp1p2a} and \eqref{eq:confomegabar}, we have  
\begin{equation}\label{eq:enconfbar}
\begin{array}{ll}
    & H(\bar\omega_k) =-n^2+n+2kn-2k^2+2k\epsilon-n\epsilon \\
    & H(\bar\omega_{n+k}) =-n^2+n+2kn-2k^2-2k\epsilon +n\epsilon 
\end{array}
\end{equation}
for any $k=0,\dots,n$. By deriving both equations in \eqref{eq:enconfbar} with respect to $k$, we have that the maxima of the energy along the path $\bar\omega$ are $H(\bar\omega_{\frac {n+\epsilon}{2}})$ and $H(\bar\omega_{n+\frac {n-\epsilon}{2}})$. This means that on the first part of the path $(\bar\omega_k)_{k=0}^n$ the maximum is reached at the critical value $k_1^*=\frac {n+\epsilon}{2}$, while on the second part of the path $(\bar\omega_{k+n})_{k=0}^n$ the maximum is reached at the critical value $k_2^*=\frac {n-\epsilon}{2}$.

Let us focus on the value $k_1^*$. Note that $H(\bar\omega_k)$ is a concave parabola in $k$, which is symmetric with respect to $k_1^*$. Since we are interested in finding the integer value of $k$ in which this maximum is achieved, we need to compare the distances $k_1^*-\lfloor k_1^*\rfloor$ and $\lceil k_1^*\rceil-k_1^*$. The minimal distance indicates the value we are interested in. Consider now the case $\epsilon\geq0$, thus
\begin{equation}
\begin{array}{ll}
\lfloor k_1^*\rfloor&=
\begin{cases}
\frac{n}{2} & \text{ if } n \text{ is even}, \\
\frac{n-1}{2} & \text{ if } n \text{ is odd and } 0\leq\epsilon<1, \\
\frac{n+1}{2} & \text{ if } n \text{ is odd and } \epsilon=1,
\end{cases} \\
\lceil k_1^*\rceil&=
\begin{cases}
\frac{n}{2}+1 & \text{ if } n \text{ is even}, \\
\frac{n+1}{2} & \text{ if } n \text{ is odd and } 0\leq\epsilon<1, \\
\frac{n+3}{2} & \text{ if } n \text{ is odd and } \epsilon=1,
\end{cases} \\
\end{array}
\end{equation}
and
\begin{equation}
\begin{array}{ll}
\lfloor k_2^*\rfloor&=
\begin{cases}
\frac{n}{2}-1 & \text{ if } n \text{ is even}, \\
\frac{n-1}{2} & \text{ if } n \text{ is odd}, 
\end{cases} \\
\lceil k_2^*\rceil&=
\begin{cases}
\frac{n}{2}& \text{ if } n \text{ is even}, \\
\frac{n+1}{2} & \text{ if } n \text{ is odd}.
\end{cases} \\
\end{array}
\end{equation}
Assume $n$ even. Since $\lfloor\frac{n+\epsilon}{2}\rfloor=\frac{n}{2}$ and $\lceil\frac{n+\epsilon}{2}\rceil=\frac{n}{2}+1$, we have that $k_1^*-\lfloor k_1^*\rfloor=\frac{\epsilon}{2}\leq1-\frac{\epsilon}{2}=\lceil k_1^*\rceil-k_1^*$ and therefore the maximum is achieved in $H(\bar\omega_{\frac{n}{2}})$. By arguing similarly for $n$ odd and $k_2^*$, we get the claim for $\epsilon\geq0$. Since $H(\bar\omega_k)=H(\hat\omega_{n+k})$ and $H(\bar\omega_{n+k})=H(\hat\omega_k)$ for any $k=0,...,n$, the case $\epsilon<0$ can be studied in a similar way. Note that for $\epsilon<0$ the values $\lfloor k_i^*\rfloor$ and $\lceil k_i^*\rceil$, with $i=1,2$, are different from the case $\epsilon>0$.
\end{proof}

\begin{proposition}[Upper bounds]\label{prop:upperbound}
Let $(\cX, Q, H, \Delta)$ be the energy landscape corresponding to the Ising model on $\mathcal{G}(2,n)$, then
$\Gamma_s\leq\Gamma^0_s$, where $\Gamma^0_s$ is defined in \eqref{eq:gamma0s}.
\end{proposition}
\begin{proof}
By using \eqref{eq:Hmin} and \cref{lemmamaxbar}, we get the claim.
\end{proof}

\subsection{Lower bounds}

For every $p \in \{0,\dots,2n\}$, define the manifold $\mathcal{C}(p) \subset \cX$ as the subset of configurations in $\cX$ with exactly $p$ plus spins, that is $\mathcal{C}(p) := \{ \sigma \in \cX ~:~ \sum_{i \in V} \mathbf{1}_{\{\sigma_i = +1 \}} = p \}$. By fixing the number of plus spins in each of the two clusters and using the notation introduced in~\cref{sub:model}, the manifold $\mathcal{C}(p)$ can be decomposed as
$$
\mathcal{C}(p) = \bigcup_{\substack{0 \leq p_1,p_2 \leq n\\p_1+p_2 = p}} C(p_1,p_2,a).
$$
Assuming the current state $\sigma \in \mathcal{C}(p)$ for some $p$, since we consider a single-flip dynamics, every nontrivial update will lead to new state $\sigma'$ that belongs to either $\mathcal{C}(p-1)$ or $\mathcal{C}(p+1)$.

\begin{figure}[!ht]
\centering
\includegraphics[width=0.5\textwidth]{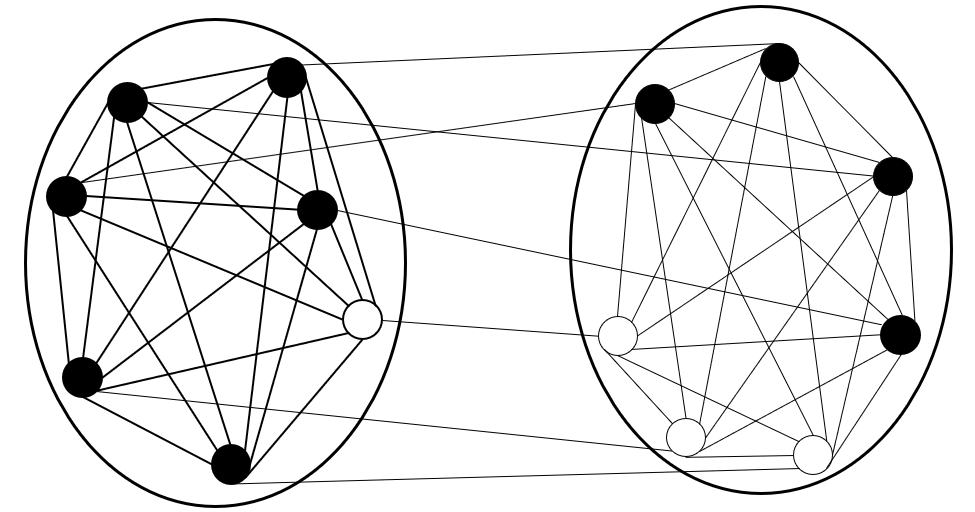}
\caption{An example of a configuration $\sigma$ on the network $\mathcal{G}(2,7)$ that belongs to the manifold $\mathcal{C}(10)$, since it has $p=10$ $+1$ spins, specifically $p_1=6$ in the first cluster and $p_2=4$ in the second cluster ($+1$/$-1$ spins are colored in black/white, respectively).}
\label{fig:manifold10}
\end{figure}

\begin{proposition}[Local minima]\label{prop:minimaonslices}
For every $n \geq 2$ and $|\epsilon|\leq 1$, regardless of the sign of $\epsilon$, the minimum value of the energy $H$ on the manifold $\mathcal{C}(p)$ is given by
\[
    H(p):=\min_{\sigma \in \mathcal{C}(p)} H(\sigma) =
    \begin{cases}
        n - (p-n)^2 -p^2- \epsilon(n-2p) & \text{ if } 0 \leq p \leq n,\\
        n - (2n-p)^2 - (p-n)^2 - \epsilon(2p-3n) & \text{ if } n \leq p \leq 2n.
    \end{cases}
\]
Furthermore, if $0 \leq p \leq n$, the minimum is achieved on the subsets $C(p,0,0)$ and $C(0,p,0)$, while if $n \leq p \leq 2n$, the minimum is achieved on the subsets $C(n,p-n,p-n)$ and $C(p-n,n,p-n)$.
\end{proposition}
\begin{proof}
For every fix $p$, one has consider all the subsets $C(p_1,p_2,a)$ which partition $\mathcal{C}(p)$. 
In view of~\eqref{eq:Hp1p2a}, we need to solve the quadratic optimization problem:
\begin{equation} \label{eq:Hp}
    \min_{\sigma \in \mathcal{C}(p)} H(\sigma) 
     = n + \epsilon (2p - n) + \min_{\substack{{p_1,p_2,a} \\
     0 \leq p_1, \, p_2 \leq n \\
     p_1+p_2 = p \\
    \max\{0,p-n\} \leq a \leq \min\{p_1,p_2\}}}
    \left( - 2 \left(p_1 - \frac{n}{2}\right)^2 - 2 \left(p_2 - \frac{n}{2}\right)^2 - 4 \epsilon a \right).
\end{equation}

If $\epsilon>0$, it is clear from~\eqref{eq:Hp} that $a$ should be as large as possible to achieve a possibly lower energy. Without loss of generality, we may assume that $p_1 \leq p_2$ and substituting $a=\min\{p_1,p_2\} = p_1$ and then $p_2 = p-p_1$, we have
\begin{equation}\label{eq:Hp2}
    \min_{\sigma \in \mathcal{C}(p)} H(\sigma) 
     = n - \epsilon (n- 2p) + \min_{\substack{{p_1}  \\
    \max\{0,p-n\} \leq p_1 \leq p/2}}
    \left( - 2 \left(p_1 - \frac{n}{2}\right)^2 - 2 \left(p-p_1 - \frac{n}{2}\right)^2 - 4 \epsilon p_1\right)
\end{equation}
Let us define $f(p_1):=\left( - 2 \left(p_1 - \frac{n}{2}\right)^2 - 2 \left(p-p_1 - \frac{n}{2}\right)^2 - 4\epsilon p_1 \right)$. Recall that $p$ is only a fixed parameter, so $f(p_1)$ single-variable concave function of $p_1$, which will then achieve its minimum value at the boundary points. The inequality $ p_1 \leq p/2$ follows from the assumptions $p_1+p_2 = p$ and $p_1 \leq p_2$. 
Recall that $p \leq 2n$ and let us distinguish two cases:
\begin{itemize}
\item[(a)] If $0 \leq p \leq n$, then the boundary points to consider are $p_1 \in \{0, \lfloor p/2 \rfloor\}$, at which the function $f(p_1)$ attains the following values
\begin{equation}
    \begin{array}{ll}
    f(0)&= -n^2-2p^2+2np, \\
    f(\lfloor\frac{p}{2}\rfloor)&= 
    \begin{cases}
    -n^2-p^2+2np-2\epsilon p & \text{ if } p \text{ is even},\\
    -n^2-p^2+2np-2\epsilon p+2\epsilon-1 & \text{ if } p \text{ is odd}.
    \end{cases}
    \end{array}
\end{equation}
By a direct computation, it follows that $f(0)\leq f(\lfloor\frac{p}{2}\rfloor)$ either whenever $\epsilon\leq\frac{p}{2}$ if $p$ is even, or whenever $\epsilon\leq\frac{p+1}{2}$ if $p$ is odd. From now on, we consider separately the three following cases. 

If $p=0$, we obtain that $f(\lfloor\frac{p}{2}\rfloor)=f(0)$ and therefore, by using \eqref{eq:Hp2},
\begin{equation}
    \min_{\sigma \in \mathcal{C}(0)} H(\sigma) = n - n^2 - \epsilon n.
\end{equation}
If $p=1$, we obtain that $f(\lfloor\frac{p}{2}\rfloor)=f(0)$ and therefore, by using \eqref{eq:Hp2},
\begin{equation}
    \min_{\sigma \in \mathcal{C}(1)} H(\sigma) = 3n - n^2 - \epsilon n + 2\epsilon -2.
\end{equation}
If $2\leq p\leq n$, since $|\epsilon|\leq1$, we have that $f(0)\leq f(\lfloor\frac{p}{2}\rfloor)$. Thus, by using \eqref{eq:Hp2},
\begin{equation}
    \min_{\sigma \in \mathcal{C}(p), \ 2\leq p\leq n} H(\sigma) = n - (p-n)^2 - p^2 - \epsilon(n-2p).
\end{equation}

\item[(b)] If $n \leq p \leq 2n$, then the boundary points to consider are $p_1 \in \{p-n, \lfloor p/2 \rfloor\}$, at which the function $f(p_1)$ attains the following values
\begin{equation}
    \begin{array}{ll}
  f(p-n)&= -5n^2-2p^2+6np-4\epsilon p + 4\epsilon n, \\
  f(\lfloor\frac{p}{2}\rfloor)&=
  \begin{cases}
  -n^2-p^2+2np-2\epsilon p & \text{ if } p \text{ is even}, \\
  -n^2-p^2+2np-2\epsilon p+2\epsilon-1 & \text{ if } p \text{ is odd}.
  \end{cases}
    \end{array}
\end{equation}
By a direct computation, it follows that $f(p-n)\leq f(\lfloor\frac{p}{2}\rfloor)$ either whenever $\epsilon\leq n-\frac{p}{2}$ if $p$ is even, or whenever $\epsilon\leq n- \frac{p-1}{2}$ if $p$ is odd. From now on, we consider separately the three following cases.

If $n\leq p \leq 2n-2$, since $|\epsilon|\leq1$, we have that $f(p-n)\leq f(\lfloor\frac{p}{2}\rfloor)$. Thus, by using \eqref{eq:Hp2},
\begin{equation}
    \min_{\sigma \in \mathcal{C}(p), \ n\leq p \leq 2n-2} H(\sigma) = n - (2n-p)^2 - (p-n)^2 - \epsilon(2p-3n).
\end{equation}
If $p=2n-1$, we obtain that $f(\lfloor\frac{p}{2}\rfloor)=f(n-1)$. Thus, by using \eqref{eq:Hp2},
\begin{equation}
\min_{\sigma \in \mathcal{C}(2n-1)} H(\sigma) = 3n - n^2 - \epsilon n + 2\epsilon -2.
\end{equation}
If $p=2n$, we obtain that $f(\lfloor\frac{p}{2}\rfloor)=f(n)$. Thus, by using \eqref{eq:Hp2},
\begin{equation}
    \min_{\sigma \in \mathcal{C}(2n)} H(\sigma) = n - n^2 - \epsilon n.
\end{equation}
\end{itemize}
From the calculations above, it is easy to deduce that if $0 \leq p \leq n$, the minimum is achieved on the subsets $C(p,0,0)$ and $C(0,p,0)$, while if $n \leq p \leq 2n$, the minimum is achieved on the subsets $C(n,p-n,p-n)$ and $C(p-n,n,p-n)$.

If $\epsilon<0$, it is clear from~\eqref{eq:Hp} that $a$ should be as small as possible to achieve a possibly lower energy. As before, without loss of generality, we assume that $p_1 \leq p_2$ and we substitute $p_2 = p-p_1$ in~\eqref{eq:Hp}. We need to distinguish two cases depending on the value of $p$.
\begin{itemize}
\item[(a)] If $0 \leq p \leq n$, then $a=\max\{0,p-n\}=0$ and~\eqref{eq:Hp} becomes
\begin{equation}
    \min_{\sigma \in \mathcal{C}(p)} H(\sigma) 
     = n - \epsilon (n - 2p) + \min_{\substack{{p_1} \\
    0 \leq p_1 \leq p/2}}
    \left( - 2 \left(p_1 - \frac{n}{2}\right)^2 - 2 \left(p-p_1 - \frac{n}{2}\right)^2\right).
\end{equation}
The objective function $g(p_1):=- 2 \left(p_1 - \frac{n}{2}\right)^2 - 2 \left(p-p_1 - \frac{n}{2}\right)^2$ is concave in $p_1$, so again we search the minimum among the boundary points $p_1 \in \{0, \lfloor p/2 \rfloor\}$, at which the function $g(p_1)$ attains the following values
\begin{equation}
    \begin{array}{ll}
    g(0)&= -n^2-2p^2+2np, \\
    g(\lfloor\frac{p}{2}\rfloor)&=
    \begin{cases}
    -n^2-p^2+2np & \text{ if } p \text{ is even}, \\
    -n^2-p^2+2np -1 & \text{ if } p \text{ is odd}.
    \end{cases}
    \end{array}
\end{equation}
By a direct computation, it follows that $g(0)\leq g(\lfloor\frac{p}{2}\rfloor)$ in both cases $p$ even and $p$ odd and, thus,
\begin{equation}
    \min_{\sigma \in \mathcal{C}(p), \, 0\leq p\leq n} H(\sigma) = n - (p-n)^2 - p^2 - \epsilon(n-2p).
\end{equation}

\item[(b)] If $n \leq p \leq 2n$, then $a=\max\{0,p-n\}=p-n$ and~\eqref{eq:Hp} becomes
\begin{equation}
    \min_{\sigma \in \mathcal{C}(p)} H(\sigma) 
     = n - \epsilon (2p-3n) + \min_{\substack{{p_1}  \\
    p-n \leq p_1 \leq p/2}} \\
    g(p_1).
\end{equation}
The objective function $g(p_1)$ is concave in $p_1$, so again we search the minimum among the boundary points $p_1 \in \{p-n,\lfloor p/2 \rfloor \}$, at which the function $g(p_1)$ attains the following values
\begin{equation}
    \begin{array}{ll}
    g(p-n)&= -5n^2 -2p^2 +6pn, \\
    g(\lfloor\frac{p}{2}\rfloor)&=
    \begin{cases}
    -n^2 - p^2 +2pn & \text{ if } p \text{ is even}, \\
    -n^2 - p^2 +2pn -1& \text{ if } p \text{ is odd}.
    \end{cases}
    \end{array}
\end{equation}
By a direct computation, it follows that $g(p-n)\leq g(\lfloor\frac{p}{2}\rfloor)$ in both cases $p$ even and $p$ odd and, thus,
\begin{equation}
    \min_{\sigma \in \mathcal{C}(p), \, n\leq p\leq 2n} H(\sigma) = n - (2n-p)^2 - (p-n)^2 - \epsilon(2p-3n).
\end{equation}
\end{itemize}
From the calculations above, it is easy to deduce that if $0 \leq p \leq n$, the minimum is achieved on the subsets $C(p,0,0)$ and $C(0,p,0)$, while if $n \leq p \leq 2n$, the minimum is achieved on the subsets $C(n,p-n,p-n)$ and $C(p-n,n,p-n)$.
\end{proof}

In order to analyze the manifold $C(p)$ with maximal energy, we need to define
\begin{equation}\label{eq:p*left}
\begin{array}{ll}
p^*_{\mathrm{left}}:=
\begin{cases}
\frac{n}{2} & \text{ if } n \text{ is even}, \\
\frac{n+1}{2} & \text{ if } n \text{ is odd and } \epsilon\geq0, \\
\frac{n-1}{2} & \text{ if } n \text{ is odd and } \epsilon<0,
\end{cases}
\end{array}
\end{equation}
and
\begin{equation}\label{eq:p*right}
\begin{array}{ll}
p^*_{\mathrm{right}}:=
\begin{cases}
n+\frac{n}{2} & \text{ if } n \text{ is even}, \\
n+\frac{n-1}{2} & \text{ if } n \text{ is odd and } \epsilon\geq0, \\
n+\frac{n+1}{2} & \text{ if } n \text{ is odd and } \epsilon<0.
\end{cases}
\end{array}
\end{equation}

For any $0 \leq p \leq 2n$, let $\mathcal{M}_p \in C(p)$ be the set of configurations with minimal energy. 

\begin{proposition}[Lower bounds]\label{prop:criticalslice}
Let $(\cX, Q, H, \Delta)$ be the energy landscape corresponding to the Ising model on $\mathcal{G}(2,n)$. The following statements hold:
\begin{itemize}
\item The maximum of the energy on $\bigcup_{0\leq p \leq n}\mathcal{M}_p$ is realized by the configurations in $C(p^*_{\mathrm{left}},0,0) \cup C(0,p^*_{\mathrm{left}},0)$;
\item The maximum of the energy on $\bigcup_{n\leq p \leq 2n}\mathcal{M}_p$ is realized by the configurations in $C(n,p^*_{\mathrm{right}}-n,p^*_{\mathrm{right}}-n) \cup C(p^*_{\mathrm{right}}-n,n,p^*_{\mathrm{right}}-n)$.
\end{itemize}

Moreover, we have that
$\Gamma_s\geq\Gamma^0_s$, where $\Gamma^0_s$ is defined in \eqref{eq:gamma0s}.
\end{proposition}
\begin{proof}
The idea of the proof is to identify, depending on the parity of $n$ and the value of $\epsilon$, the correct manifold that would give the desired lower bound.

Treating $H(p)$ as a function of a continuous variable, we see that is concave and, solving for $\frac{d}{d p} H(p) = 0$, we obtain two stationary points $p_{\text{left}}=\frac{n}{2} + \frac{\epsilon}{2}$ and $p_{\text{right}}=\frac{3n}{2} - \frac{\epsilon}{2}$. They both yield the value 
\[
    \max_{0 \leq p \leq 2n} H(p) = -\frac{1}{2} \left(n^2-2n -\epsilon ^2\right).
\] 
Since $p_{\text{left}}$ and $p_{\text{right}}$ can only take integer values, we deduce that the possible integer optimal values are 
\[
  p^*_1 \in  \Big\{\Big\lfloor \frac{n}{2} + \frac{\epsilon}{2} \Big\rfloor, \Big\lceil \frac{n}{2} + \frac{\epsilon}{2} \Big\rceil\Big\}, \quad p^*_2 \in \Big\{\Big\lfloor \frac{3n}{2} -\frac{\epsilon}{2} \Big\rfloor, \Big\lceil \frac{3n}{2} -\frac{\epsilon}{2} \Big\rceil \Big\}.
\]
By performing the same computations as in the proof of \cref{lemmamaxbar}, we obtain that $p^*_1=p^*_{\text{left}}$ and $p^*_2=p^*_{\text{right}}$, where $p^*_{\text{left}}$ (resp.\ $p^*_{\text{right}}$) is defined in \eqref{eq:p*left} (resp.\  \eqref{eq:p*left}).
Furthermore, by \cref{prop:minimaonslices} we have that the minimum of the energy on the manifold $C(p^*_{\text{left}})$ is realized in $\mathcal{M}_{p^*_{\text{left}}} \equiv C(p^*_{\text{left}},0,0) \cup C(0,p^*_{\text{left}},0)$ and on the manifold $C(p^*_{\text{right}})$ in $\mathcal{M}_{p^*_{\text{right}}}\equiv C(n,p^*_{\text{right}}-n,p^*_{\text{right}}-n) \cup C(p^*_{\text{right}}-n,n,p^*_{\text{right}}-n)$.

\end{proof}

\begin{corollary}[Maximal energy barrier]\label{clr:gamma}
We have that
\[
\Gamma_s =
\begin{cases}
     \frac{n^2}{2} + |\epsilon| n &\text{ if $n$ is even}, \\
     \frac{n^2-1}{2} + |\epsilon| (n + 1) &\text{ if $n$ is odd}.
    \end{cases}
\]
\end{corollary}

\begin{proof}
We get the claim by combining \cref{prop:upperbound} and \cref{prop:criticalslice}.
\end{proof}

\subsection{Identification of stable and metastable states}

In this section, we provide the proof of \cref{thm:metastabstates}. To this end, we give two propositions that allow us to accomplish this task. The proof of \cref{prop:stable} and \cref{prop:riducibilita} are postponed after the proof of \cref{thm:metastabstates}.

\begin{proposition}[Identification of stable states]\label{prop:stable}
Let $(\cX, Q, H, \Delta)$ be the energy landscape corresponding to the Ising model on $\mathcal{G}(2,n)$. Then, the lowest possible value of the energy is equal to
\[
    \min_{\sigma \in \cX} H(\sigma) = -n^2 + n -|\epsilon| n,
\]
and the set of stable states is
\[
    \cX_s=
    \begin{cases}
       \{\ppone, \mmone\} & \text{ if } \epsilon>0,\\
       \{\ppone, \mmone,\pmone, \mpone\} & \text{ if } \epsilon = 0,\\
       \{\pmone, \mpone\} & \text{ if } \epsilon<0.
    \end{cases}
\]
\end{proposition}

\begin{proposition}[Identification of metastable states]\label{prop:riducibilita}
Let $\sigma\in\cX\setminus\{\ppone,\pmone,\mpone,\mmone\}$, then the stability level of $\sigma$ is zero, i.e., $V_{\sigma}=0$. The set of metastable states is
\[
    \cX_m=
    \begin{cases}
       \{\pmone, \mpone\} & \text{ if } \epsilon>0,\\
       \{\ppone, \mmone\} & \text{ if } \epsilon<0.
    \end{cases}
\]
Moreover, we have that
\begin{equation}
 \Gamma_s = 
    \begin{cases}
        \frac{n^2}{2} - |\epsilon| n &\text{ if $n$ is even}, \\
        \frac{n^2-1}{2}-|\epsilon|(n-1) &\text{ if $n$ is odd},
    \end{cases}
\end{equation}
and 
\begin{equation}\label{eq:metagamma}
    \Gamma_m = 
    \begin{cases}
        \frac{n^2}{2} - |\epsilon| n &\text{ if $n$ is even}, \\
        \frac{n^2-1}{2}-|\epsilon|(n-1) &\text{ if $n$ is odd}.
    \end{cases}
\end{equation}
\end{proposition}
\begin{proof}[Proof of \cref{thm:metastabstates}]
Combining \cref{clr:gamma}, \cref{prop:stable} and \cref{prop:riducibilita} we get the claim.
\end{proof}

\begin{proof}[Proof of \cref{prop:stable}]
Recalling that $\max\{p_1+p_2-n,0\}\leq a \leq \min\{p_1,p_2\}$, we note that $a$ is a function of $p_1$ and $p_2$. 
In view of the partition 
\[\cX=\displaystyle\bigcup_{\substack{0 \leq p_1,p_2 \leq n \\ \max\{0,p_1+p_2-n\} \leq a \leq \min\{p_1,p_2\}}} C(p_1,p_2,a)
\]
and~\eqref{eq:Hp1p2a}, we can calculate the minimum energy as
\begin{align*}
    \min_{\substack{{p_1,\, p_2} 
                                         }} H(p_1,p_2,a) 
                                     &= n - n\epsilon
                                     + 2\min_{\substack{{p_1,\, p_2} 
                                         }}\left( -  \left(p_1 - \frac{n}{2}\right)^2 -  \left(p_2 - \frac{n}{2}\right)^2 + \epsilon (p_1 + p_2)
                                         -2\epsilon a \right) \\
                                        &=: n - n\epsilon
                                     + 2\min_{\substack{{p_1,\, p_2}}} f(p_1,p_2).
\end{align*}
If $\epsilon\geq0$, we have that
\[\min_{\substack{{p_1,\, p_2}}} f(p_1,p_2)=\min_{\substack{{p_1,\, p_2}}}\left( -  \left(p_1 - \frac{n}{2}\right)^2 -  \left(p_2 - \frac{n}{2}\right)^2 +  \epsilon (p_1 + p_2)-2\epsilon \min\{p_1,p_2\} \right),
\]
so the function $f(p_1,p_2)$ is concave in both variables. Thus, we expect the minimum $(p_1^*, p_2^*)$ to be achieved at the boundary of the feasible region. This immediately implies that $(p_1^*, p_2^*) \in \{(0,0),(0,n),(n,0),(n,n)\}$.  By direct computation, we obtain:
\begin{equation}\label{eq:confronto*}
f(0,0)=f(n,n)=-\frac{n^2}{2}; \quad f(0,n)=f(n,0)=-\frac{n^2}{2}+n\epsilon.
\end{equation}
This implies that the minimum is achieved at $(p_1^*,p_2^*)=(0,0)$ and $(p_1^*,p_2^*)=(n,n)$, which correspond to the configuration $C(0,0,0)\equiv\mmone$ and $C(n,n,n)\equiv\ppone$, respectively.

If $\epsilon<0$, we have that
\[\min_{\substack{{p_1,\, p_2}}} f(p_1,p_2)=\min_{\substack{{p_1,\, p_2}}}\left( -  \left(p_1 - \frac{n}{2}\right)^2 -  \left(p_2 - \frac{n}{2}\right)^2 + \epsilon (p_1 + p_2)-2\epsilon \max\{p_1+p_2-n,0\} \right),
\]
so the function $f(p_1,p_2)$ is concave in both variables as before. Thus, we deduce that the possible configurations in which the minimum is achieved are the same as in \eqref{eq:confronto*}. By direct computation, the minimum is attained at $(p_1^*,p_2^*)=(n,0)$ and $(p_1^*,p_2^*)=(0,n)$, which correspond to the configuration $C(n,0,0)\equiv\pmone$ and $C(0,n,0)\equiv\mpone$, respectively.
\end{proof}

\begin{proof}[Proof of \cref{prop:riducibilita}]
Consider a configuration $\sigma\in C(p_1,p_2,a)$, with $0\leq p_1,p_2 \leq n$ and $\max\{p_1+p_2-n,0\}\leq a \leq \min\{p_1,p_2\}$. Note that such a configuration $\sigma$ can communicate via one step of the dynamics with a configuration $\sigma'$ such that
\begin{equation}\label{eq:sigma'0}
\sigma'\in
    \begin{cases}
      C(p_1+1,p_2,a) & \text{ if } p_1\neq n \text{ and } a>\max\{p_1+p_2-n,0\}, \\
       C(p_1,p_2+1,a) & \text{ if } p_2\neq n \text{ and } a>\max\{p_1+p_2-n,0\}, \\
       C(p_1+1,p_2,a+1) & \text{ if } p_1\neq n \text{ and } a=\max\{p_1+p_2-n,0\}, \\
       C(p_1,p_2+1,a+1) & \text{ if } p_2\neq n \text{ and } a=\max\{p_1+p_2-n,0\}, \\
       C(p_1-1,p_2,a) & \text{ if } p_1\neq0 \text{ and } a<\min\{p_1,p_2\} \text{ or } p_1>p_2 \text{ and } a=\min\{p_1,p_2\}, \\
       C(p_1,p_2-1,a) & \text{ if } p_2\neq0 \text{ and } a<\min\{p_1,p_2\} \text{ or } p_2>p_1 \text{ and } a=\min\{p_1,p_2\},\\
       C(p_1-1,p_2,a-1) & \text{ if } p_1\neq0, \, p_1\leq p_2 \text{ and }a=\min\{p_1,p_2\}, \\
       C(p_1,p_2-1,a-1) & \text{ if } p_2\neq0, \, p_2\leq p_1 \text{ and }a=\min\{p_1,p_2\}.
    \end{cases}
\end{equation}
In other words, $\sigma'$ is a configuration obtained from $\sigma$ via either an up-flip or a down-flip in one of the two clusters. First, we will prove that if $\sigma\in C(p_1,p_2,a) \setminus\{\mmone,\mpone,\pmone,\ppone\}$, then $H(\sigma')-H(\sigma)<0$, with $\sigma'$ one of the configurations described in \eqref{eq:sigma'0}. To this end, we consider the following cases.
\begin{itemize}
    \item[A.] $p_1=n$ and $a\geq\max\{p_1+p_2-n,0\}$;
    \item[B.] $p_1\neq n$ and $a>\max\{p_1+p_2-n,0\}$;
    \item[C.] $p_1\neq n$ and $a=\max\{p_1+p_2-n,0\}$.
\end{itemize}

\noindent
{\bf Case A.} Since it is not possible to have $p_1=n$ and $a>\max\{p_1+p_2-n,0\}$, we note that now $\sigma\in C(n, p_2, p_2)$. Since $\sigma\notin\{\ppone,\pmone\}$, it follows that $0<p_2<n$. By using \cref{lmm:1}, we deduce that 
\begin{equation}\label{eq:confronto40}
H(C(n,p_2+1,p_2+1))-H(C(n,p_2,p_2))<0 
\quad \Longleftrightarrow \ p_2\geq \Big\lceil\frac{n-1}{2}-\frac{\epsilon}{2}\Big\rceil. 
\end{equation}
Thus, if $p_2$ satisfies \eqref{eq:confronto40}, then we are done. Otherwise, by using \cref{lmm:2} we deduce that $H(C(n,p_2-1,p_2-1))-H(C(n,p_2,p_2))<0$. 

\noindent   
{\bf Case B.} By using \cref{lmm:1}, we deduce that 
\begin{equation}\label{eq:confronto10}
H(C(p_1+1,p_2,a))-H(C(p_1,p_2,a))<0 \ \Longleftrightarrow \ p_1\geq \Big\lceil \frac{n-1}{2}+\frac{\epsilon}{2}\Big\rceil. 
\end{equation}
Thus, if $p_1$ satisfies \eqref{eq:confronto10}, then we are done. Otherwise, we argue as follows. First, we note that the case $p_1=0$ implies $a=0$, but this case is not allowed since $a>\max\{p_1+p_2-n,0\}$.

If $p_1>p_2$, we get $H(\sigma')-H(\sigma)<0$ with $\sigma'$ belonging to $C(p_1-1,p_2,a)$. Indeed, by using \cref{lmm:2}, we have that 
\begin{equation}\label{eq:confronto20}
H(C(p_1-1,p_2,a))-H(C(p_1,p_2,a))<0,  
\end{equation}
since $p_1\leq \Big\lfloor \frac{n-1}{2}+\frac{\epsilon}{2}\Big\rfloor$.

If $p_1\leq p_2$, we get $H(\sigma')-H(\sigma)<0$ with $\sigma'$ belonging to $C(p_1-1,p_2,a-1)$. Indeed, by using \cref{lmm:2}, we have that
\begin{equation}\label{eq:confronto30}
H(C(p_1-1,p_2,a-1))-H(C(p_1,p_2,a))<0 
\end{equation}
since $p_1\leq \Big\lfloor\frac{n-1}{2}+\frac{\epsilon}{2}\Big\rfloor$.

\noindent
{\bf Case C.} First of all, we note that if $p_2=n$ then we repeat the argument as in case A. Thus, we assume $p_2 \neq n$. By using \cref{lmm:1}, we deduce that 
\begin{equation}\label{eq:confronto10*}
H(C(p_1+1,p_2,a+1))-H(C(p_1,p_2,a))<0 \ \Longleftrightarrow \ p_1\geq \Big\lceil \frac{n-1}{2}-\frac{\epsilon}{2}\Big\rceil,
\end{equation}
\begin{equation}\label{eq:confronto11*}
H(C(p_1,p_2+1,a+1))-H(C(p_1,p_2,a))<0 \ \Longleftrightarrow \ p_2\geq \Big\lceil \frac{n-1}{2}-\frac{\epsilon}{2}\Big\rceil. 
\end{equation}
Thus, if $p_1$ satisfies \eqref{eq:confronto10*} or $p_2$ satisfies \eqref{eq:confronto11*}, then we are done. Otherwise, $a=\max\{p_1+p_2-n,0\}=0$ and we have $p_1 \neq 0$ or $p_2 \neq 0$ since $\sigma \neq \mmone$. Without loss of generality, we suppose $p_1 \neq 0$ and we apply \cref{lmm:2}. We obtain 
\begin{equation}\label{eq:confronto12*}
H(C(p_1-1,p_2,a))-H(C(p_1,p_2,a))<0,
\end{equation}
since $p_1\leq \Big\lfloor \frac{n-1}{2}-\frac{\epsilon}{2}\Big\rfloor$.

Thus, we proved that the stability level for every configuration different from $\{\mmone,\mpone,\pmone,\ppone\}$ is zero. It remains to show that $\cX_m=\{\pmone,\mpone\}$ (resp.\ $\cX_m=\{\mmone,\ppone\}$) if $0<\epsilon\leq1$ (resp.\ $-1\leq\epsilon<0$) and to compute the maximal stability level $\Gamma_m$. In the case $\epsilon=0$, all these states have the same energy and therefore there is no metastable state. 

In the case $\epsilon>0$, we have $\cX_s=\{\mmone,\ppone\}$. By considering the part of the path $\bar\omega:\mmone\rightarrow\ppone$ defined in \eqref{eq:confomegabar} connecting $\pmone$ to $\ppone$, and by using \eqref{eq:phibaromega}, we deduce that 
$$
\Gamma_m\leq
\begin{cases}
\frac{n^2}{2}-\epsilon n &\text{if $n$ is even}, \\
\frac{n^2-1}{2}-\epsilon(n-1) &\text{if $n$ is odd}.
\end{cases}
$$
In order to prove also the reverse inequality, we argue as in the proof of~\cite[eq.\ (3.86)]{Nardi2005}. 
The case $\epsilon<0$ can be treated in a similar way. Thus we get the claim.
\end{proof}

\subsection{Asymptotic behavior of the tunneling time}
In this section, we prove Theorem \ref{thm:tunneling}. Recalling \eqref{eq:metagamma}, we observe that in all above cases $\Gamma_s - \Gamma_m = 2n |\epsilon| >0$ in the case $\epsilon\neq0$, which means that the corresponding energy landscape exhibits the absence of deep cycles. In the case $\epsilon=0$, we deduce that $\Gamma_s - \Gamma_m=0$, indeed all the states $\{\ppone,\pmone,\mpone,\mmone\}$ are stable. Thanks to \cite[Lemma 3.6]{Nardi2015}, we deduce that for our model the quantity $\tilde\Gamma(B)$, with $B\subsetneq\cX$, defined in \cite[eq.\ (21)]{Nardi2015} is such that $\tilde\Gamma(\cX\setminus\{s_2\})=\Gamma_s$. Moreover, thanks to the property of absence of deep cycles, \cite[Proposition 3.18]{Nardi2015} implies that $\Theta(s_1,s_2)=\Gamma_s$ for $s_1,s_2\in\cX_s$. Thus, \cref{thm:tunneling}(i) follows from \cite[Corollary 3.16]{Nardi2015}. Moreover, \cref{thm:tunneling}(ii) follows from \cite[Theorem 3.17]{Nardi2015} provided that \cite[Assumption A]{Nardi2015} is satisfied: this is implied by the absence of deep cycles and \cite[Proposition 3.18]{Nardi2015}. Finally, \cref{thm:tunneling}(iii) follows from \cite[Theorem 3.19]{Nardi2015} provided that \cite[Assumption B]{Nardi2015} is satisfied: this is implied by the absence of deep cycles and the argument carried out in \cite[Example 4]{Nardi2015}. \cref{thm:tunneling}(iv) follows from \cite[Proposition 3.24]{Nardi2015} with $\tilde\Gamma(\cX\setminus\{s_2\})=\Gamma_s$ for any $s_2\in\cX_s$.

\subsection{Gate for the tunneling transition}
In this section, we prove Theorem \ref{thm:gate}. If $0\leq\epsilon\leq1$ (resp.\ $-1\leq\epsilon<0$), consider an optimal path $\omega\in(\mmone\rightarrow\ppone)_{opt}$ (resp.\ $\omega\in(\pmone\rightarrow\mpone)_{opt}$). Since any path from $\mmone$ to $\ppone$ (resp.\ from $\pmone$ to $\mpone$) has to cross each manifold $C(p)$, with $0\leq p\leq 2n$ (resp.\ either $0\leq p \leq n$ or $n\leq p\leq 2n$), and due to the optimality of the path $\omega$, by \cref{prop:minimaonslices} and \cref{prop:criticalslice} we get the claim.


\section{Proof of the main results: case \texorpdfstring{$h>0$}{h>0}}
\label{sec:proofh}

\subsection{Reference paths}

If $\epsilon\geq0$, consider the path $\bar\omega$ defined in \cref{defbaromega}.

\begin{figure}
\includegraphics[width=\textwidth]{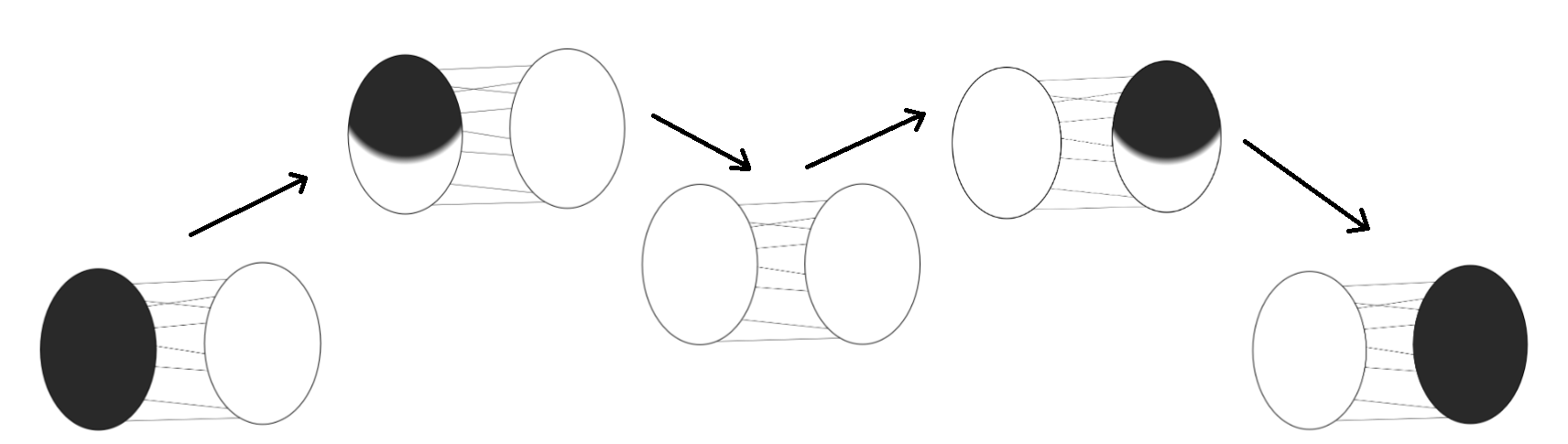}
\caption{Here we depict the reference path $\check\omega$ by representing the saddles, the metastable and stable states that it crosses, where we represent in white (resp.\ black) the minus (resp.\ plus) spins.}
\label{fig:pathhat2}
\end{figure}

\begin{definition}[Reference paths]\label{deftildeomega} 
If $0<h<-\epsilon\leq1$, we define $\check\omega:\pmone\to\mpone$ as the path $(\check\omega_k)_{k=0}^{2n}$, with
\begin{align}\label{eq:hatconfomega}
    \check\omega_k\in C(n-k,0,0)\ \text{ and }\ \check\omega_{n+k}\in C(0,k,0),  \quad \text{for any } k=0,\dots,n.
\end{align}
If $0<-\epsilon<h\leq1$, we define $\tilde\omega:\pmone\to\ppone$ as the path $(\tilde\omega_k)_{k=0}^{n}$, with
\begin{align}\label{eq:tildeconfomega}
    \tilde\omega_k\in C(n,k,k), \quad \text{for any } k=0,\dots,n.
\end{align} 
\end{definition}
\noindent
See \cref{fig:pathhat2} (resp.\ \cref{fig:pathtilde}) to visualize the reference path $\check\omega$ (resp.\ $\tilde\omega$).
\begin{figure}[!htb]
\centering
\includegraphics[width=0.6\textwidth]{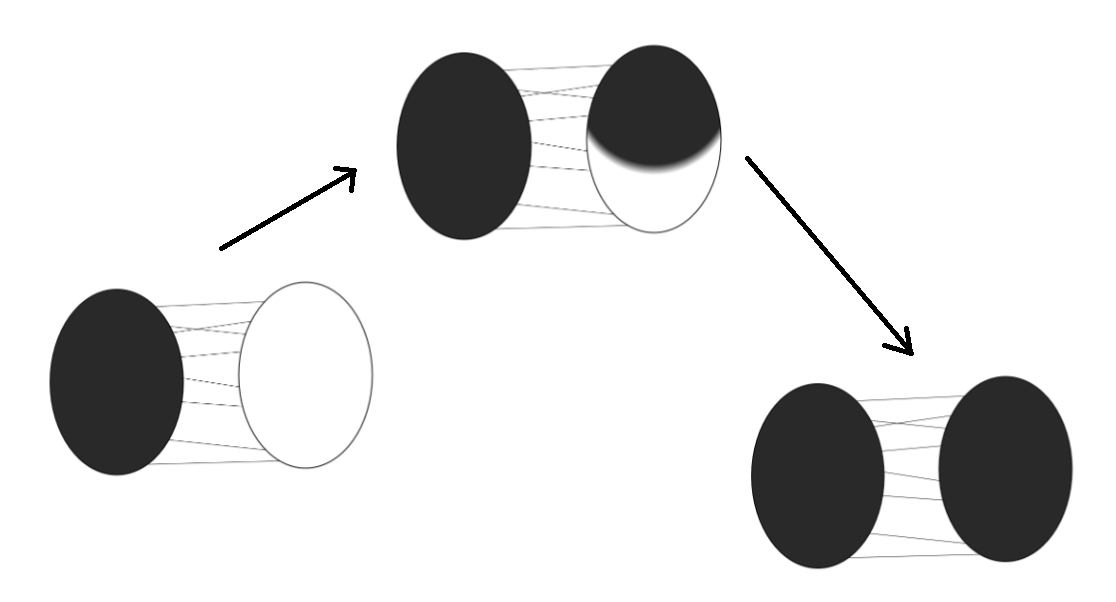}
    \caption{Here we depict the reference path $\tilde\omega$ by representing the saddles, the metastable and stable states that it crosses, where we represent in white (resp.\ black) the minus (resp.\ plus) spins.}
    \label{fig:pathtilde}
\end{figure}

\begin{lemma}[Maximal energy along the reference paths]\label{lemmamaxtilde} 
If $\epsilon\geq0$, let $\bar\omega:\mmone\to\ppone$ the path defined in \cref{defbaromega}. Then
\begin{align}\label{eq:phibaromegah}
    \Phi_{\bar\omega}=
    \begin{cases}
    H(\bar\omega_{\frac{n}{2}})=n-\frac{n^2}{2}+hn &\text{if $n$ is even}, \\
    H(\bar\omega_{\frac{n+1}{2}})=n-\frac{n^2+1}{2}+\epsilon+h(n-1) &\text{if $n$ is odd and } 0<h\leq\epsilon\leq1, \\
    H(\bar\omega_{\frac{n-1}{2}})=n-\frac{n^2+1}{2}-\epsilon+h(n+1) &\text{if $n$ is odd and } 0\leq\epsilon<h\leq1.
    \end{cases}
\end{align}
If $0<h<-\epsilon\leq1$, let $\check\omega:\pmone\to\mpone$ be the path given in \eqref{eq:hatconfomega}. Then,
\begin{align}
    \Phi_{\check\omega}=
    \begin{cases}
    H(\check\omega_{\frac{n}{2}})=H(\check\omega_{n+\frac{n}{2}})=n-\frac{n^2}{2}+hn &\text{if $n$ is even and } 0<h-\epsilon<1, \\
    H(\check\omega_{\frac{n+2}{2}})=H(\check\omega_{n+\frac{n-2}{2}})=n-\frac{n^2}{2}-2(\epsilon+1)+h(n+2) &\text{if $n$ is even and } 1\leq h-\epsilon<2, \\
    H(\check\omega_{\frac{n+1}{2}})=H(\check\omega_{n+\frac{n-1}{2}})=n-\frac{n^2+1}{2}-\epsilon+h(n+1) &\text{if $n$ is odd}.
    \end{cases}
\end{align}
If $0<-\epsilon<h\leq1$, let $\tilde\omega:\pmone\to\ppone$ be the path given in \eqref{eq:tildeconfomega}. Then, 
\begin{align}\label{eq:phitildeomega}
    \Phi_{\tilde\omega}=
    \begin{cases}
    H(\tilde\omega_{\frac{n}{2}})=n-\frac{n^2}{2}-hn &\text{if $n$ is even}, \\
    H(\tilde\omega_{\frac{n-1}{2}})=n-\frac{n^2+1}{2}+\epsilon-h(n-1) &\text{if $n$ is odd}.
    \end{cases}
\end{align}
\end{lemma}
\begin{proof}
From \eqref{eq:Hp1p2a} and \eqref{eq:confomegabar}, we have  
\begin{equation}\label{eq:enconftilde}
\begin{array}{ll}
    & H(\bar\omega_k) =-n^2+n+2kn-2k^2+2k\epsilon-n\epsilon -2h(k-n), \\
    & H(\bar\omega_{n+k}) =-n^2+n+2kn-2k^2-2k\epsilon +n\epsilon -2hk,
\end{array}
\end{equation}
for any $k=0,\dots,n$. By deriving both equations in \eqref{eq:enconftilde} with respect to $k$, we have that the maxima of the energy along the path $\bar\omega$ are $H(\bar\omega_{\frac {n+\epsilon-h}{2}})$ and $H(\bar\omega_{n+\frac {n-\epsilon-h}{2}})$. This means that on the first part of the path $(\bar\omega_k)_{k=0}^n$ the maximum is reached at the critical value $k_1^*=\frac {n+\epsilon-h}{2}$, while on the second part of the path $(\bar\omega_{k+n})_{k=0}^n$ the maximum is reached at the critical value $k_2^*=\frac {n-\epsilon-h}{2}$.

First, consider the case $0<h\leq\epsilon\leq1$. Let us focus on the value $k_1^*$. Note that $H(\bar\omega_k)$ is a concave parabola in $k$, which is symmetric with respect to $k_1^*$. Since we are interested in finding the integer value of $k$ in which this maximum is achieved, we need to compare the distances $k_1^*-\lfloor k_1^*\rfloor$ and $\lceil k_1^*\rceil-k_1^*$. The minimal distance indicates the value we are interested in. Since $0\leq\epsilon-h<1$, we have that
\begin{equation}
\begin{array}{ll}
\lfloor k_1^*\rfloor&=
\begin{cases}
\frac{n}{2} & \text{ if } n \text{ is even}, \\
\frac{n-1}{2} & \text{ if } n \text{ is odd},
\end{cases} \\
\lceil k_1^*\rceil&=
\begin{cases}
\frac{n}{2}+1 & \text{ if } n \text{ is even}, \\
\frac{n+1}{2} & \text{ if } n \text{ is odd},
\end{cases} \\
\end{array}
\end{equation}
and
\begin{equation}
\begin{array}{ll}
\lfloor k_2^*\rfloor&=
\begin{cases}
\frac{n}{2}-1 & \text{ if } n \text{ is even}, \\
\frac{n-1}{2} & \text{ if } n \text{ is odd and } 0<\epsilon+h\leq1, \\ 
\frac{n-3}{2} & \text{ if } n \text{ is odd and } 1<\epsilon+h\leq2,
\end{cases} \\
\lceil k_2^*\rceil&=
\begin{cases}
\frac{n}{2}& \text{ if } n \text{ is even}, \\
\frac{n+1}{2} & \text{ if } n \text{ is odd and } 0<\epsilon+h\leq1, \\ 
\frac{n-1}{2} & \text{ if } n \text{ is odd and } 1<\epsilon+h\leq2.
\end{cases} \\
\end{array}
\end{equation}
Assume $n$ even. Since $\lfloor\frac{n+\epsilon-h}{2}\rfloor=\frac{n}{2}$ and $\lceil\frac{n+\epsilon-h}{2}\rceil=\frac{n}{2}+1$, we have that $k_1^*-\lfloor k_1^*\rfloor=\frac{\epsilon-h}{2}\leq1-\frac{\epsilon-h}{2}=\lceil k_1^*\rceil-k_1^*$ and therefore the maximum is achieved in $H(\bar\omega_{\frac{n}{2}})$. By arguing similarly for $n$ odd and for $k_2^*$, we get the claim.

Consider now the case $0\leq\epsilon<h\leq1$. By arguing as before, we get the claim.

Consider now the case $0<h<-\epsilon\leq1$. From \eqref{eq:Hp1p2a} and \eqref{eq:hatconfomega}, we have  
\begin{equation}\label{eq:hatenconf}
\begin{array}{ll}
   & H(\check\omega_k) =-n^2+n+2kn-2k^2-2k\epsilon+n\epsilon + 2hk, \\
    & H(\check\omega_{n+k}) =-n^2+n+2kn-2k^2+2k\epsilon -n\epsilon -2h(k-n),
\end{array}
\end{equation}
for any $k=0,\dots,n$. By deriving both equations in \eqref{eq:hatenconf} with respect to $k$, we have that the maxima of the energy along the path $\check\omega$ are $H(\check\omega_{\frac {n-\epsilon+h}{2}})$ and $H(\check\omega_{n+\frac{n+\epsilon-h}{2}})$. This means that on the first part of the path $(\check\omega_k)_{k=0}^n$ the maximum is reached at the critical value $k_1^*=\frac {n-\epsilon+h}{2}$, while on the second part of the path $(\check\omega_{k+n})_{k=0}^n$ the maximum is reached at the critical value $k_2^*=\frac {n+\epsilon-h}{2}$. We have that
\begin{equation}
\begin{array}{ll}
\lfloor k_1^*\rfloor&=
\begin{cases}
\frac{n}{2} & \text{ if } n \text{ is even}, \\
\frac{n-1}{2} & \text{ if } n \text{ is odd and } 0<h-\epsilon<1, \\
\frac{n+1}{2} & \text{ if } n \text{ is odd and } 1\leq h-\epsilon<2, 
\end{cases} \\
\lceil k_1^*\rceil&=
\begin{cases}
\frac{n+2}{2} & \text{ if } n \text{ is even}, \\
\frac{n+1}{2} & \text{ if } n \text{ is odd and } 0<h-\epsilon<1, \\
\frac{n+3}{2} & \text{ if } n \text{ is odd and } 1\leq h-\epsilon<2, 
\end{cases} \\
\end{array}
\end{equation}
and
\begin{equation}
\begin{array}{ll}
\lfloor k_2^*\rfloor&=
\begin{cases}
\frac{n-2}{2} & \text{ if } n \text{ is even}, \\
\frac{n-1}{2} & \text{ if } n \text{ is odd and } 0<h-\epsilon\leq1, \\
\frac{n-3}{2} & \text{ if } n \text{ is odd and } 1\leq h-\epsilon<2,
\end{cases} \\
\lceil k_2^*\rceil&=
\begin{cases}
\frac{n}{2}& \text{ if } n \text{ is even}, \\
\frac{n+1}{2} & \text{ if } n \text{ is odd and } 0<h-\epsilon\leq1, \\
\frac{n-1}{2} & \text{ if } n \text{ is odd and } 1\leq h-\epsilon<2.
\end{cases} \\
\end{array}
\end{equation}
By arguing as above, we get the claim.

Consider now the case $0<-\epsilon<h\leq1$. From \eqref{eq:Hp1p2a} and \eqref{eq:tildeconfomega}, we have  
\begin{equation}\label{eq:tildeenconf}
H(\tilde\omega_k)=-n^2+n+\epsilon n -2k^2 +2nk -2\epsilon k -2hk.
\end{equation}
By deriving the equation in \eqref{eq:tildeenconf} with respect to $k$, we have that the maximum of the energy along the path $\tilde\omega$ is $H(\tilde\omega_{\frac {n-\epsilon-h}{2}})$. By arguing as before, we get the claim.
\end{proof}

\begin{proposition}[Upper bounds]\label{prop:upperboundh}
Let $(\cX, Q, H, \Delta)$ be the energy landscape corresponding to the Ising model on $\mathcal{G}(2,n)$. In the case $0\leq\epsilon\leq1$, we have $\Gamma_m\leq \Gamma^1_m$, where $\Gamma^1_m$ is defined in \eqref{eq:gamma1m}. In the case $0<-\epsilon<h\leq1$, we have $\Gamma_m\leq\Gamma^2_m$, where $\Gamma^2_m$ is defined in \eqref{eq:gamma2m}. In the case $0<h<-\epsilon\leq1$, we have $\Gamma_s\leq\Gamma^h_s$, where $\Gamma^h_s$ is defined in \eqref{eq:gammahs}.
\end{proposition}
\begin{proof}
By using \eqref{eq:Hmin2} and \cref{lemmamaxtilde}, we get the claim.
\end{proof}

\subsection{Lower bounds}

\begin{proposition}[Local minima]\label{prop:minimaonslicesh}
For every $n \geq 2$ and $|\epsilon|\leq1$, regardless the sign of $\epsilon$, the minimum value of the energy $H$ on the manifold $\mathcal{C}(p)$ is given by
\[
    H(p):=\min_{\sigma \in \mathcal{C}(p)} H(\sigma) =
    \begin{cases}
        n - (p-n)^2 -p^2- \epsilon(n-2p)- 2h(p-n) & \text{ if } 0 \leq p \leq n,\\
        n - (2n-p)^2 - (p-n)^2 - \epsilon(2p-3n)- 2h(p-n) & \text{ if } n \leq p \leq 2n.
    \end{cases}
\]
Furthermore, if $0 \leq p \leq n$, the minimum is achieved on the subsets $C(p,0,0)$ and $C(0,p,0)$, while if $n \leq p \leq 2n$, the minimum is achieved on the subsets $C(n,p-n,p-n)$ and $C(p-n,n,p-n)$.
\end{proposition}

\begin{proof}
Note that on the manifold $C(p)$, with $0\leq p\leq 2n$, the energy contribution of the external magnetic field is equal to $-2h(p-n)$, which is constant. Thus the claim simply follows by \cref{prop:minimaonslices} by adding this further term to the energy.
\end{proof}

In order to analyze the manifold $C(p)$ with maximal energy, we need to define
\begin{equation}\label{eq:p*1}
p^*_1:=
\begin{cases}
\frac{n}{2} & \text{ if } n \text{ is even}, \\
\frac{n+1}{2} & \text{ if } n \text{ is odd and } 0<h\leq\epsilon\leq1, \\
\frac{n-1}{2} & \text{ if } n \text{ is odd and } 0\leq\epsilon<h\leq1,
\end{cases}
\end{equation}
and
\begin{equation}\label{eq:p*2}
p^*_2:=
\begin{cases}
\frac{3n}{2} &\text{if $n$ is even}, \\
\frac{3n-1}{2} &\text{if $n$ is odd},
\end{cases}
\end{equation}
and
\begin{equation}\label{eq:p*3}
p^*_3:=
\begin{cases}
\frac{n}{2} &\text{if $n$ is even and } 0<h-\epsilon<1, \\
\frac{n-2}{2} &\text{if $n$ is even and } 1\leq h-\epsilon<2, \\
\frac{n-1}{2} &\text{if $n$ is odd}.
\end{cases}
\end{equation}

In the following proposition, depending on the values of the parameters $\epsilon$ and $h$, we calculate the maximum of the energy over different collections of manifolds $\mathcal{M}_p$, since the relevant starting and target configurations are not always $\ppone$ and $\mmone$.

\begin{proposition}[Lower bounds]\label{prop:criticalsliceh}
Let $(\cX, Q, H, \Delta)$ be the energy landscape corresponding to the Ising model on $\mathcal{G}(2,n)$. The following statements hold:
\begin{itemize}
\item If $0\leq\epsilon\leq1$, the maximum of the energy on $\bigcup_{0\leq p \leq 2n}\mathcal{M}_p$ is realized by the configurations in $C(p^*_1,0,0)\cup C(0,p^*_1,0)$. Moreover, we have that $\Gamma_m\geq \Gamma^1_m$, where $\Gamma^1_m$ is defined in \eqref{eq:gamma1m};
\item If $0<-\epsilon<h\leq1$, the maximum of the energy on $\bigcup_{n\leq p \leq 2n}\mathcal{M}_p$ is realized by the configurations in $C(n,p^*_2-n,p^*_2-n)\cup C(p^*_2-n,n,p^*_2-n)$. Moreover, we have that $\Gamma_m\geq\Gamma^2_m$, where $\Gamma^2_m$ is defined in \eqref{eq:gamma2m};
\item If $0<h<-\epsilon\leq1$, the maximum of the energy on $\bigcup_{0\leq p \leq 2n}\mathcal{M}_p$ is realized by the configurations in $C(p^*_3,0,0)\cup C(0,p^*_3,0)$. Moreover, we have that $\Gamma_s\geq\Gamma^h_s$, where $\Gamma^h_s$ is defined in \eqref{eq:gammahs}.
\end{itemize}
\end{proposition}
\begin{proof}
The idea of the proof is to identify, depending on the parity of $n$ and the values of $\epsilon$ and $h$, the correct manifold that would give the desired lower bound.

Treating $H(p)$ as a function of a continuous variable, we see that is concave and, solving for $\frac{d}{d p} H(p) = 0$, we obtain two stationary points $p_{\text{left}}=\frac{n}{2} + \frac{\epsilon-h}{2}$ and $p_{\text{right}}=\frac{3n}{2} - \frac{\epsilon+h}{2}$. Since $p_{\text{left}}$ and $p_{\text{right}}$ can only take integer values, we deduce that the possible integer optimal values are
\[
    p^*_{\text{left}} \in \Big\{\Big\lfloor \frac{n}{2} + \frac{\epsilon-h}{2} \Big\rfloor, \Big\lceil \frac{n}{2} + \frac{\epsilon-h}{2} \Big\rceil \Big\},
    \quad p^*_{\text{right}} \in \Big\{\Big\lfloor \frac{3n}{2} - \frac{\epsilon+h}{2} \Big\rfloor, \Big\lceil \frac{3n}{2} - \frac{\epsilon+h}{2} \Big\rceil \Big\}.
\]
If $0\leq\epsilon\leq1$ (resp.\ $0<h<-\epsilon\leq1$), since the path from $m\in\cX_m$ to $s\in\cX_s$ (resp.\ from $s_1\in\cX_s$ to $s_2\in\cX_s$) has to cross each manifold $C(p)$, with $0\leq p\leq 2n$, we need to take into account both $p^*_{\text{left}}$ and $p^*_{\text{right}}$. We have that
\[
\begin{array}{ll}
H(C(\frac{n}{2}+\frac{\epsilon-h}{2}))&=n-\frac{n^2}{2}+nh+\frac{1}{2}(\epsilon-h)^2, \\
H(C(\frac{3n}{2}-\frac{\epsilon+h}{2}))&=n-\frac{n^2}{2}-nh+\frac{1}{2}(\epsilon+h)^2.
\end{array}
\] 
By direct computation, we deduce that the maximum is reached in $H(C(\frac{n}{2}+\frac{\epsilon-h}{2}))$. By performing the same computations in the proof of \cref{lemmamaxtilde}, we obtain that $p^*_{\text{left}}=p^*_1$ (resp.\ $p^*_{\text{left}}=p^*_3$) in the case $0\leq\epsilon\leq1$ (resp.\ $0<h<-\epsilon\leq1$), where $p^*_1$ (resp.\ $p^*_3$) is defined in \eqref{eq:p*1} (resp.\ \eqref{eq:p*3}). Furthermore, by \cref{prop:minimaonslicesh} we have that the minimum of the energy on the manifold $C(p^*_1)$ (resp.\ $C(p^*_3)$) is realized in $\mathcal{M}_{p_1^*}\equiv C(p_1^*,0,0) \cup C(0,p_1^*,0)$ (resp.\ in $\mathcal{M}_{p_3^*}\equiv C(p_3^*,0,0) \cup C(0,p_3^*,0)$)
if $0\leq\epsilon\leq1$ (resp.\ $0<h<-\epsilon\leq1$). 

Consider now the case $0<-\epsilon<h\leq1$. In this case, since for any $m\in\{\pmone,\mpone\}$ we are interested in the transition from $m$ to $\ppone$, we have that every path connecting these two states crosses the foliations $C(p)$ with $n\leq p\leq 2n$. Thus in this case we have that the critical value of $p$ is 
\[
p^*_{\text{right}}\in\Big\{\Big\lfloor \frac{3n}{2} - \frac{\epsilon+h}{2}\Big\rfloor, \Big\lceil \frac{3n}{2} - \frac{\epsilon+h}{2}\Big\rceil\Big\}.
\]
By performing the same computations in the proof of \cref{lemmamaxtilde}, we obtain that $p^*_{\text{right}}=p^*_2$, where $p^*_2$ is defined in \eqref{eq:p*2}. Furthermore, by \cref{prop:minimaonslicesh} we have that the minimum of the energy on the manifold $C(p^*_2)$ is realized in $\mathcal{M}_{p_2^*}\equiv C(n,p^*_2-n,p^*_2-n)\cup C(p^*_2-n,n,p^*_2-n)$. Thus we get the claim.
\end{proof}

\begin{corollary}[Maximal energy barrier]\label{clr:gammah}
Let $(\cX, Q, H, \Delta)$ be the energy landscape corresponding to the Ising model on $\mathcal{G}(2,n)$. If $0\leq\epsilon\leq1$, we have that 
\begin{align}
        \Gamma_m=
        \begin{cases}
        \frac{n^2}{2}+n(\epsilon-h) &\text{if $n$ is even}, \\
        \frac{n^2-1}{2}+(n+1)(\epsilon-h) & \text{if $n$ is odd and } 0<h\leq\epsilon\leq1, \\
        \frac{n^2-1}{2}+(n-1)(\epsilon-h) & \text{if $n$ is odd and } 0\leq\epsilon<h\leq1.
        \end{cases}
    \end{align}
If $0<-\epsilon<h\leq1$, we have that
    \begin{equation}
    \Gamma_m=
    \begin{cases}
    \frac{n^2}{2}-n(\epsilon+h) &\text{if $n$ is even}, \\
    \frac{n^2-1}{2}-(n-1)(\epsilon+h) &\text{if $n$ is odd}.
    \end{cases}
    \end{equation}
If $0<h<-\epsilon\leq1$, we have that
    \begin{equation}
    \Gamma_s=
    \begin{cases}
    \frac{n^2}{2}-n(\epsilon+h) &\text{if $n$ is even}, \\
    \frac{n^2-1}{2}-(n+1)(\epsilon+h) &\text{if $n$ is odd}.
    \end{cases}
    \end{equation}
\end{corollary}
\begin{proof}
We get the claim by combining \cref{prop:upperboundh} and \cref{prop:criticalsliceh}.
\end{proof}

\subsection{Identification of metastable and stable states}
In this section, we provide the proof of \cref{thm:metstatesh}. To this end, we give two propositions that allow us to accomplish this task. The proof of \cref{prop:stableh} and \cref{prop:riducibilitah} are postponed after the proof of \cref{thm:metstatesh}.

\begin{proposition}[Identification of stable states]\label{prop:stableh}
Let $(\cX, Q, H, \Delta)$ be the energy landscape corresponding to the Ising model on $\mathcal{G}(2,n)$. If $0\leq\epsilon\leq1$, Then, the lowest possible energy is equal to
\begin{equation}
    \min_{\sigma \in \cX} H(\sigma) =
    \begin{cases}
    -n^2 + n -\epsilon n -2hn & \text{if } 0\leq\epsilon\leq1 \text{ or } 0<-\epsilon<h\leq1, \\  
     -n^2 + n +\epsilon n  & \text{if } 0<h\leq-\epsilon\leq1,
    \end{cases}
\end{equation}
and the set of stable states is
\[
    \cX_s=
    \begin{cases}
       \{\ppone\} & \text{if } 0\leq\epsilon\leq1 \text{ or } 0<-\epsilon<h\leq1,\\
       \{\ppone, \pmone, \mpone\} & \text{if } h=-\epsilon, \\
       \{\pmone, \mpone\} & \text{if } 0<h<-\epsilon\leq1.
    \end{cases}
\]
\end{proposition}

\begin{proposition}[Identification of metastable states]\label{prop:riducibilitah}
Let $(\cX, Q, H, \Delta)$ be the energy landscape corresponding to the Ising model on $\mathcal{G}(2,n)$. Let $\sigma\in\cX\setminus\{\ppone,\pmone,\mpone,\mmone\}$, then the stability level of $\sigma$ is zero, i.e., $V_{\sigma}=0$. Furthermore, the set of metastable states is 
\[
\cX_m=
\begin{cases}
\{\mmone\} & \text{if } 0\leq\epsilon\leq1 \text{ or } h=-\epsilon, \\
\{\pmone,\mpone\} & \text{if } 0<-\epsilon<h\leq1, \\
\{\ppone\} & \text{if } 0<h<-\epsilon\leq1.
\end{cases}
\]
Moreover, in the case $0\leq\epsilon\leq1$, we have that
\begin{equation}\label{eq:gammapippo}
\Gamma_m =
\begin{cases}
    \frac{n^2}{2}+n(\epsilon-h) &\text{ if $n$ is even}, \\
    \frac{n^2-1}{2}+(n+1)(\epsilon-h) &\text{ if $n$ is odd and } 0<h\leq\epsilon\leq1, \\
    \frac{n^2-1}{2}+(n-1)(\epsilon-h) &\text{ if $n$ is odd and } 0\leq\epsilon<h\leq1,
    \end{cases}
\end{equation}
whereas in the case $0<-\epsilon<h\leq1$, 
we have that
\begin{equation}\label{eq:gammapippo2}
\Gamma_m =
\begin{cases}
    \frac{n^2}{2}-n(\epsilon+h) &\text{ if $n$ is even}, \\
    \frac{n^2-1}{2}-(n-1)(\epsilon+h) &\text{ if $n$ is odd},
    \end{cases}
\end{equation}
and in the case $0<h<-\epsilon\leq1$, 
we have that
\begin{equation}
\Gamma_s =
\begin{cases}
     \frac{n^2}{2}+n(h-\epsilon)&\text{ if $n$ is even and }0<h-\epsilon<1, \\
     \frac{n^2-4}{2}+(n+2)(h-\epsilon)&\text{ if $n$ is even and }1\leq h-\epsilon<2, \\
     \frac{n^2-1}{2}+(n+1)(h-\epsilon)&\text{ if $n$ is odd},
    \end{cases}
\end{equation}
and
\begin{equation}\label{eq:gammapippo3}
    \Gamma_m=
    \begin{cases}
    \frac{n^2}{2}+n(\epsilon+h) &\text{ if $n$ is even}, \\
        \frac{n^2-1}{2}+(n-1)(\epsilon+h) &\text{ if $n$ is odd}.
    \end{cases}
\end{equation}
\end{proposition}

\begin{proof}[Proof of \cref{thm:metstatesh}]
Combining \cref{clr:gammah}, \cref{prop:stableh} and \cref{prop:riducibilitah} we get the claim.
\end{proof}

\begin{proof}[Proof of \cref{prop:stableh}.]
Recalling that $\max\{p_1+p_2-n,0\}\leq a \leq \min\{p_1,p_2\}$, we note that $a$ is a function of $p_1$ and $p_2$. 
In view of the partition 
\[\cX=\displaystyle\bigcup_{\substack{0 \leq p_1,p_2 \leq n \\ \max\{0,p_1+p_2-n\} \leq a \leq \min\{p_1,p_2\}}} C(p_1,p_2,a)
\]
and~\eqref{eq:Hp1p2a}, we can compute the minimum energy as
\begin{align*}
    \min_{\substack{{p_1,\, p_2} 
                                         }} H(p_1,p_2,a) 
                                     &= n - n(\epsilon - 2h)
                                     + 2\min_{\substack{{p_1,\, p_2} 
                                         }}\left( -  \left(p_1 - \frac{n}{2}\right)^2 -  \left(p_2 - \frac{n}{2}\right)^2 +  (\epsilon-h) (p_1 + p_2)
                                         -2\epsilon a \right) \\
                                        &=: n - n(\epsilon - 2h)
                                     + 2\min_{\substack{{p_1,\, p_2}}} f(p_1,p_2).
\end{align*}
If $\epsilon\geq0$, we have that
\[\min_{\substack{{p_1,\, p_2}}} f(p_1,p_2)=\min_{\substack{{p_1,\, p_2}}}\left( -  \left(p_1 - \frac{n}{2}\right)^2 -  \left(p_2 - \frac{n}{2}\right)^2 +  (\epsilon-h) (p_1 + p_2)-2\epsilon \min\{p_1,p_2\} \right),
\]
so the function $f(p_1,p_2)$ is concave in both variables. Thus, we expect the minimum $(p_1^*, p_2^*)$ to be achieved at the boundary of the feasible region. This immediately implies that $(p_1^*, p_2^*) \in \{(0,0),(0,n),(n,0),(n,n)\}$. By direct computation, we obtain:
\begin{equation}\label{eq:confronto}
f(0,0)=-\frac{n^2}{2}; \quad f(0,n)=f(n,0)=-\frac{n^2}{2}+n(\epsilon-h); \quad f(n,n)=-\frac{n^2}{2}-2hn.
\end{equation}
This implies that the minimum is achieved at $(p_1^*,p_2^*)=(n,n)$, which corresponds to the configuration $C(n,n,n)\equiv\ppone$, as claimed.

If $\epsilon<0$, we have that
\[\min_{\substack{{p_1,\, p_2}}} f(p_1,p_2)=\min_{\substack{{p_1,\, p_2}}}\left( -  \left(p_1 - \frac{n}{2}\right)^2 -  \left(p_2 - \frac{n}{2}\right)^2 +  (\epsilon-h) (p_1 + p_2)-2\epsilon \max\{p_1+p_2-n,0\} \right),
\]
so the function $f(p_1,p_2)$ is concave in both variables as before. Thus, we deduce that the possible configurations in which the minimum is achieved are the same as in \eqref{eq:confronto}. By direct computation, the minimum is attained either at $(p_1^*,p_2^*)=(n,n)$ whenever $h>-\epsilon$, which corresponds to the configuration $C(n,n,n)\equiv\ppone$, or at $(p_1^*,p_2^*)=(n,0)$ and $(p_1^*,p_2^*)=(0,n)$ whenever $h<-\epsilon$, which corresponds to the configurations $C(n,0,0)\equiv\pmone$ and $C(0,n,0)\mpone$. In the special case $h=-\epsilon$, all these configurations realize the minimum of the energy. This concludes the proof.
\end{proof}

\begin{proof}[Proof of \cref{prop:riducibilitah}.]
Let $0<\epsilon\leq1$. Consider a configuration $\sigma\in C(p_1,p_2,a)$, with $0\leq p_1,p_2 \leq n$ and $\max\{p_1+p_2-n,0\}\leq a \leq \min\{p_1,p_2\}$. Note that such a configuration $\sigma$ can communicate via one step of the dynamics with a configuration $\sigma'$ as in \eqref{eq:sigma'0}.
In other words, $\sigma'$ is a configuration obtained from $\sigma$ via either an up-flip or a down-flip in one of the two clusters. First, we will prove that if $\sigma\in C(p_1,p_2,a) \setminus\{\mmone,\mpone,\pmone,\ppone\}$, then $H(\sigma')-H(\sigma)<0$, with $\sigma'$ one of the configurations described in \eqref{eq:sigma'0}. To this end, we consider the following cases.
\begin{itemize}
    \item[A.] $p_1=n$ and $a\geq\max\{p_1+p_2-n,0\}$;
    \item[B.] $p_1\neq n$ and $a>\max\{p_1+p_2-n,0\}$;
    \item[C.] $p_1\neq n$ and $a=\max\{p_1+p_2-n,0\}$.
\end{itemize}
{\bf Case A.} Since it is not possible to have $p_1 = n$ and $a > \max\{p_1 + p_2 - n, 0\}$, we note that now $\sigma \in C(n, p_2, p_2)$. Since $\sigma\notin\{\ppone, \pmone\}$, it follows that $0 < p_2 < n$. By using \cref{lmm:1}, we deduce that 

\begin{equation}\label{eq:pippo}
H(C(n,p_2+1,p_2+1))-H(C(n,p_2,p_2))<0 \ \Longleftrightarrow \ p_2\geq \Big\lceil \frac{n-1}{2}-\frac{\epsilon+h}{2} \Big\rceil.
\end{equation}
Thus, if $p_2$ satisfies \eqref{eq:pippo}, then we are done. Otherwise, by using \cref{lmm:2} we deduce that $H(C(n,p_2-1,p_2-1))-H(C(n,p_2,p_2))<0$.

\noindent
{\bf Case B.} By using \cref{lmm:1}, we deduce that 
\begin{equation}\label{eq:pippo1}
H(C(p_1+1,p_2,a))-H(C(p_1,p_2,a))<0 \ \Longleftrightarrow \ p_1\geq \Big\lceil \frac{n-1}{2}+\frac{\epsilon-h}{2} \Big\rceil.
\end{equation}
Thus, if $p_1$ satisfies \eqref{eq:pippo1}, then we are done. Otherwise, we argue as follows. First, we note that the case $p_1=0$ implies $a=0$, but this case is not allowed since $a>\max\{p_1+p_2-n,0\}$.

If $p_1>p_2$, we get $H(\sigma')-H(\sigma)<0$ with $\sigma'$ belonging to $C(p_1-1,p_2,a)$. Indeed, by using \cref{lmm:2}, we have that
\begin{equation}
H(C(p_1-1,p_2,a))-H(C(p_1,p_2,a))<0,
\end{equation}
since $p_1\leq\lfloor\frac{n-1}{2}+\frac{\epsilon-h}{2}\rfloor$.

If $p_1\leq p_2$, we get $H(\sigma')-H(\sigma)<0$ with $\sigma'$ belonging to $C(p_1-1,p_2,a-1)$. Indeed, by using \cref{lmm:2}, we have that
\begin{equation}
H(C(p_1-1,p_2,a-1))-H(C(p_1,p_2,a))<0,
\end{equation}
since $p_1\leq\lfloor\frac{n-1}{2}+\frac{\epsilon-h}{2}\rfloor$.

\medskip
\noindent
{\bf Case C.} First of all, we note that if $p_2=n$, then we repeat the argument as in case A. Thus, we assume $p_2\neq n$. By using \cref{lmm:1}, we deduce that
\begin{equation}\label{eq:pluto1}
H(C(p_1+1,p_2,a+1))-H(C(p_1,p_2,a))<0 \ \Longleftarrow p_1\leq\Big\lceil\frac{n-1}{2}-\frac{\epsilon+h}{2}\Big\rceil,
\end{equation}
\begin{equation}\label{eq:pluto2}
H(C(p_1,p_2+1,a+1))-H(C(p_1,p_2,a))<0 \ \Longleftarrow p_2\leq\Big\lceil\frac{n-1}{2}-\frac{\epsilon+h}{2}\Big\rceil.
\end{equation}
Thus, if $p_1$ satisfies \eqref{eq:pluto1} or $p_2$ satisfies \eqref{eq:pluto2}, then we are done. Otherwise, $a=\max\{p_1+p_2-n,0\}=0$ and we have $p_1\neq0$ or $p_2\neq0$, since $\sigma\neq\mmone$. Without loss of generality, we suppose $p_1\neq0$ and we apply \cref{lmm:2}. We obtain 
\begin{equation}
H(C(p_1-1,p_2,a))-H(C(p_1,p_2,a))<0,
\end{equation}
since $p_1\leq\lfloor\frac{n-1}{2}-\frac{\epsilon+h}{2}\rfloor$.

Thus, we have proven that the stability level for every configuration $\sigma\notin\{\mmone,\mpone,\pmone,\ppone\}$ is zero. It remains to identify the set of metastable states and to compute their stability level $\Gamma_m$. 

In the case $0\leq\epsilon\leq1$, we have that $\cX_s=\{\ppone\}$. By considering the path $\bar\omega:\mmone\rightarrow\ppone$ defined in \eqref{eq:confomegabar}, by using \eqref{eq:phibaromega}, we deduce that 
\begin{equation}
\Phi_{\bar\omega}(\pmone,\ppone)-H(\pmone)<\Phi_{\bar\omega}(\mmone,\ppone)-H(\mmone)\leq
\Gamma_m,
\end{equation}
where $\Gamma_m$ is as in \eqref{eq:gammapippo}. In order to prove also the reverse inequality, we argue as in the proof of~\cite[eq.\ (3.86)]{Nardi2005}. Thus $\cX_m=\{\mmone\}$.

In the case $0<-\epsilon<h\leq1$, we have that $\cX_s=\{\ppone\}$. By arguing as above, we deduce that now 
\begin{equation}
\Phi_{\bar\omega}(\mmone,\ppone)-H(\mmone)<\Phi_{\bar\omega}(\pmone,\ppone)-H(\pmone)\leq\Gamma_m,
\end{equation}
where $\Gamma_m$ is as in \eqref{eq:gammapippo2}. Thus $\cX_m=\{\pmone,\mpone\}$.

In the case $0<h<-\epsilon\leq1$, we have that $\cX_s=\{\pmone,\mpone\}$. By considering the part of the path $\check\omega$ defined in \eqref{eq:hatconfomega} connecting $\mmone$ to $\mpone$, and defining the path $\omega^*=(\omega^*_1,...,\omega^*_n):\ppone\rightarrow\mpone$ as $\omega^*_k\in C(n-k,n,n-k)$ for $k=0,...,n$, we deduce that
\begin{equation}
\Phi_{\check\omega}(\mmone,\mpone)-H(\mmone)<\Phi_{\omega^*}(\ppone,\mpone)-H(\ppone)\leq \Gamma_m,
\end{equation}
where $\Gamma_m$ is as in \eqref{eq:gammapippo3}. To prove also the reverse inequality, we argue as in the proof of~\cite[eq.\ (3.86)]{Nardi2005}. Thus $\cX_m=\{\ppone\}$.
\end{proof}

\subsection{Asymptotic behavior of the tunneling time}
In this section, we prove Theorem \ref{thm:tunnelingtimeh}. We recall \eqref{eq:gammapippo3}. Note that in the case $0<h<-\epsilon\leq1$, for which we are interested in studying the tunneling time for the transition from $s_1$ to $s_2$, where $s_1,s_2\in\{\pmone,\mpone\}$, we have that 
\[
\Gamma_s - \Gamma_m = 
\begin{cases}
-2n\epsilon &\text{ if $n$ is even and } 0<h-\epsilon<1, \\
-2(1+n\epsilon -h+\epsilon) &\text{ if $n$ is even and } 1\leq h-\epsilon<1, \\
-2(n\epsilon-h) &\text{ if $n$ is odd}.
\end{cases}
\]
In all the above cases we have that $\Gamma_s - \Gamma_m >0$ since $\epsilon<0$, which means that the corresponding energy landscape exhibits the absence of deep cycles.

Here we are interested in the tunneling time from $s_1$ to $s_2$ for any $s_1,s_2\in\cX_s$. Thanks to \cite[Lemma 3.6]{Nardi2015}, we deduce that for our model the quantity $\tilde\Gamma(B)$, with $B\subsetneq\cX$, defined in \cite[eq.\ (21)]{Nardi2015} is such that $\tilde\Gamma(\cX\setminus\{s_2\})=\Gamma_s$. Moreover, thanks to the property of absence of deep cycles, \cite[Proposition 3.18]{Nardi2015} implies that $\Theta(s_1,s_2)=\Gamma_s$ for $s_1,s_2\in\cX_s$. Thus, \cref{thm:tunnelingtimeh}(i) follows from \cite[Corollary 3.16]{Nardi2015}. Moreover, \cref{thm:tunnelingtimeh}(ii) follows from \cite[Theorem 3.17]{Nardi2015} provided that \cite[Assumption A]{Nardi2015} is satisfied: this is implied by the absence of deep cycles and \cite[Proposition 3.18]{Nardi2015}. Finally, \cref{thm:tunnelingtimeh}(iii) follows from \cite[Theorem 3.19]{Nardi2015} provided that \cite[Assumption B]{Nardi2015} is satisfied: this is implied by the absence of deep cycles and the argument carried out in \cite[Example 4]{Nardi2015}. \cref{thm:tunnelingtimeh}(iv) follows from \cite[Proposition 3.24]{Nardi2015} with $\tilde\Gamma(\cX\setminus\{s_2\})=\Gamma_s$ for any $s_2\in\cX_s$.

\subsection{Asymptotic behavior of the transition time}
In this section, we prove Theorem \ref{thm:transitiontime}. Here we are interested in studying the transition from a metastable to a stable state. Thus, \cref{thm:transitiontime} follows from \cite[Theorems 4.1, 4.9 and 4.15]{Manzo2004} together with \cref{thm:metstatesh} and \cref{clr:gammah}. \cref{thm:transitiontime}(iv) follows from \cite[Proposition 3.24]{Nardi2015} with $\tilde\Gamma(\cX\setminus\{s\})=\Gamma_m$ for any $s\in\cX_s$.

\subsection{Gate for the transition}
In this section, we prove Theorem \ref{thm:gateh}. If $0\leq\epsilon\leq1$, consider $\omega\in(\mmone\rightarrow\ppone)_{opt}$. Since any path from $\mmone$ to $\ppone$ has to cross each manifold $C(p)$ with $0\leq p\leq 2n$, and due to the optimality of the path $\omega$, by \cref{prop:minimaonslicesh} and \cref{prop:criticalsliceh} we get the claim.

If $0<-\epsilon<h\leq1$, consider either $\omega\in(\pmone\rightarrow\ppone)_{opt}$ or $\omega\in(\mpone\rightarrow\ppone)_{opt}$. Since any path from either $\pmone$, or $\mpone$, to $\ppone$ has to cross each manifold $C(p)$ with $n\leq p\leq 2n$, and due to the optimality of the path $\omega$, by \cref{prop:minimaonslicesh} and \cref{prop:criticalsliceh} we get the claim.

If $0<h<-\epsilon\leq1$, consider $\omega\in(\pmone\rightarrow\mpone)_{opt}$. Since any path from $\pmone$ to $\mpone$ has to cross each manifold $C(p)$ with $0\leq p\leq n$, due to the optimality of the path $\omega$, by \cref{prop:minimaonslicesh} and \cref{prop:criticalsliceh} we get the claim.

\section{Conclusions and future work}
\label{sec:conclusions}

We investigated opinion dynamics inside a community of individuals via the analysis of metastability for the Ising model on the graph $\mathcal{G}(2,n)$. Depending on the different parameters $\epsilon$ and $h$, we showed that the stable and metastable states of the system are different. Thus, according to the different scenarios, we used the framework of the pathwise approach \cite{Manzo2004,Nardi2015} to analyze the transition time or tunneling time, respectively, and to describe the critical configurations. Moreover, we showed that the presence of a positive external magnetic field, which can be interpreted as external information or influence, makes the situation much richer, especially in the case $\epsilon<0$ in which communities tend to have diverging opinions. Thus, the set of stable states is completely different according to the role given to the external information with respect to influence between communities, namely depending on whether $h<-\epsilon$ or not. This model is our first attempt to analyze the spread of an opinion inside two communities. First, the extension to a general number $k$ of communities naturally arises in this context and will be the focus of future work, together with the computation of the prefactor for the mean transition time. This represents a challenging task in the case $k>2$, as one needs to take into account all the mechanisms of spreading the new opinion among different communities. Further, one may consider models with more than two opinions (Potts model) or with different interaction strengths among communities. We believe that the opinion dynamics inside a population of individuals with a nontrivial network topology is a topic of great interest with many several interesting directions to explore further in future research work. \medskip \\

\noindent \textbf{Acknowledgments:} S.B, A.G., and V.J.~are grateful for the support of ``Gruppo Nazionale per l'Analisi Matematica e le loro Applicazioni" (GNAMPA-INdAM).

\bibliographystyle{abbrv}
\bibliography{references.bib}

\end{document}